\documentclass[12pt]{amsart}
\usepackage{amscd,amssymb}
\usepackage[mathscr]{eucal}
\usepackage[english]{babel}
\usepackage{tikz}
\usepackage{comment}

\usepackage{color}
\usepackage{pstricks, pst-node}
\input xy
\usepackage[all]{xy}

\usepackage{amsmath}
   \usepackage{ucs}
   \usepackage{amsmath}
   \usepackage{amsfonts}
   \usepackage{amssymb}
   \usepackage{yhmath}

\usepackage[latin1]{inputenc}%
\usepackage{a4}

\newlength{\itemlaenge}

\newtheoremstyle{mytheorem}
  {}
  {}
  {\slshape}
  {}
  {\scshape}
  {.}
  { }
  {}

\newtheoremstyle{mydefinition}
  {}
  {}
  {\upshape}
  {}
  {\scshape}
  {.}
  { }
  {}

\theoremstyle{mytheorem}
\newtheorem{lemma}{Lemma}[section]
\newtheorem{prop}[lemma]{Proposition}
\newtheorem*{prop*}{Proposition}

\newtheorem{thm}[lemma]{Theorem}

\newtheorem*{thm*}{Theorem}
\newtheorem*{claim}{Claim}
\theoremstyle{mydefinition}
\newtheorem{rem}[lemma]{Remark}
\newtheorem*{rem*}{Remark}

\newtheorem*{notation*}{Notation}

\newtheorem*{warning*}{Warning}

\newtheorem*{defi*}{Definition}

\newtheorem*{question}{Question}

\numberwithin{equation}{section}

\newcommand{\bqn}{\begin{equation*}}
\newcommand{\eqn}{\end{equation*}}
\newcommand{\bq}{\begin{equation}}
\newcommand{\eq}{\end{equation}}
\newcommand{\ba}{\begin{aligned}}
\newcommand{\ea}{\end{aligned}}
\newcommand{\be}{\begin{enumerate}}
\newcommand{\ee}{\end{enumerate}}

\newcommand{\thismonth}{\ifcase\month 
  \or January\or February\or March\or April\or May\or June%
  \or July\or August\or September\or October\or November%
  \or December\fi}

\newcommand{\Stab}{\operatorname{Stab}}

\newcommand{\SL}{\operatorname{SL}}

\newcommand{\SO}{\operatorname{SO}}
\newcommand{\GL}{\operatorname{GL}}

\newcommand{\DD}{{\mathbb D}}

\newcommand{\HH}{{\mathbb H}}

\newcommand{\RR}{{\mathbb R}}
\newcommand{\TT}{{\mathbb T}}
\newcommand{\ZZ}{{\mathbb Z}}

\newcommand{\e}{\epsilon}

\def\h{{\rm H}}
\def\hb{{\rm H}_{\rm b}}

\def\hc{{\rm H}_{\rm c}}
\def\hcb{{\rm H}_{\rm cb}}

\def\one{\mathbf{1\kern-1.6mm 1}}

\def\id{{\it I\! d}}

\def\c{{\operatorname{c}}}

\def\h2{{\operatorname{H_2}}}
\def\h1{{\operatorname{H_1}}}

\def\id{{\operatorname{Id}}}

\def\Vol{{\operatorname{vol}}}
\def\Vol{{\operatorname{Vol}}}

\def\SL{\operatorname{SL}}

\def\SO{\operatorname{SO}}

\def\det{{\operatorname{det}}}

\def\hcb{{\rm H}_{\rm cb}}
\def\to{\rightarrow}

\def\hb{{\rm H}_{\rm b}}
\def\hc{{\rm H}_{\rm c}}
\def\h{{\rm H}}

\def\hhcb{\hat{\rm H}_{\rm cb}}
\def\hb{{\rm H}_{\rm b}}
\def\hc{{\rm H}_{\rm c}}
\def\h{{\rm H}}

\def\hhb{{\rm H}_{\rm B,b}}
\def\hhc{{\rm H}_{\rm B}}
\def\hhcb{{\rm H}_{\rm B,b}}

\renewcommand{\phi}{\varphi}

\def\No{N\raise4pt\hbox{\tiny o}\kern+.2em}
\def\no{n\raise4pt\hbox{\tiny o}\kern+.2em}

\newcommand{\rot}{\mathrm{Rot}}

\renewcommand{\O}{\textup{O}}
\newcommand{\rep}{\textup{Hom}}

\renewcommand{\hom}{\textup{Hom}}

\newcommand{\isom}{\mathrm{Isom}}
\newcommand{\isomp}{\mathrm{Isom}^+}

\newcommand\ep{\epsilon}
\newcommand{\Rep}{\RR_\epsilon}
\newcommand{\Zep}{\ZZ_\epsilon}
\newcommand{\omb}{\omega_{2m}^{\mathrm{b}}}
\newcommand{\ve}{\varepsilon}
\newcommand{\vemb}{\ve_{2m}^{\mathrm{b}}}

\newcommand{\vem}{\varepsilon_{2m}}
\newcommand{\vei}{\varepsilon_{(i)}}
\newcommand{\veib}{\varepsilon_{(i)}^{\mathrm{b}}}

\DeclareFontFamily{U}{matha}{\hyphenchar\font45}
\DeclareFontShape{U}{matha}{m}{n}{
      <5> <6> <7> <8> <9> <10> gen * matha
      <10.95> matha10 <12> <14.4> <17.28> <20.74> <24.88> matha12
      }{}
\DeclareSymbolFont{matha}{U}{matha}{m}{n}
\DeclareFontSubstitution{U}{matha}{m}{n}
\DeclareMathSymbol{\abxcup}{\mathbin}{matha}{'131}
\begin{document}

\title{Integrality of Volumes of Representations}

\author[M.~Bucher]{Michelle Bucher}
\address{Section de Math\'ematiques Universit\'e de Gen\`eve, 
2-4 rue du Li\`evre, Case postale 64, 1211 Gen\`eve 4, Suisse}
\email{Michelle.Bucher-Karlsson@unige.ch}

\author[M.~Burger]{Marc Burger}
\address{Department Mathematik, ETH Z\"urich, 
R\"amistrasse 101, CH-8092 Z\"urich, Switzerland}
\email{burger@math.ethz.ch}

\author[A.~Iozzi]{Alessandra Iozzi}
\address{Department Mathematik, ETH Z\"urich, 
R\"amistrasse 101, CH-8092 Z\"urich, Switzerland}
\email{iozzi@math.ethz.ch}

\thanks{Michelle Bucher was partially supported by Swiss National Science Foundation 
projects PP00P2-128309/1, 200020-178828/1 and 200021-169685, Alessandra Iozzi was partial supported by the 
Swiss National Science Foundation projects 2000021-127016/2 and 200020-144373
and Marc Burger was partially supported by the Swiss National Science Foundation project 200020-144373.  
The authors thank the Institute Mittag-Leffler in Djursholm, Sweden, 
and the Institute for Advances Studies in Princeton, NJ, for their warm hospitality during  the preparation of this paper.}


\date{\today}

\begin{abstract} Let $M$ be an oriented complete hyperbolic $n$-manifold of finite volume.
Using the definition of volume of a representation previously given by the authors in \cite{Bucher_Burger_Iozzi_Mostow}
we show that the volume of a representation $\rho\colon \pi_1(M)\to\isomp(\HH^n)$, properly normalized, 
takes integer values if $n$ is even and $\geq4$. 

If $M$ is not compact and $3$-dimensional, it is known that the volume is not locally constant.
In this case we give explicit examples of representations with volume as arbitrary as the volume
of hyperbolic manifolds obtained from $M$ via Dehn fillings.

\end{abstract}
\maketitle
%
%
%
%
\section{Introduction}\label{sec:intro}
Let $M$ be a connected oriented complete hyperbolic manifold of finite volume,
which we represent as the quotient $M=\Gamma\backslash\HH^n$ of real hyperbolic $n$-space $\HH^n$
by a torsion-free lattice $\Gamma<\isomp(\HH^n)$ in the group of orientation preserving isometries of $\HH^n$.

Given a representation $\rho\colon \Gamma\to\isomp(\HH^n)$, our central object of study is the volume $\Vol(\rho)$
of $\rho$ as defined in \cite{Burger_Iozzi_Wienhard_toledo} for $n=2$ and in general in \cite{Bucher_Burger_Iozzi_Mostow}.
This notion extends the classical one introduced in \cite{Goldman82} for $M$ compact
and, as it was shown in \cite{Kim_Kim_OnDeformation}, if $M$ is of finite volume it coincides with definitions introduced by other authors
\cite{Dunfield, Francaviglia, Kim_Kim_Volume}.

We refer to \S~\ref{sec:congruence} for the definition of $\Vol(\rho)$ and content ourselves with listing some of its main properties.

\be
\item The volume function is uniformly bounded on the representation variety $\rep(\Gamma,\isomp(\HH^n))$, that is
\bqn
|\Vol(\rho)|\leq\Vol(M)
\eqn
and 
\bqn
\Vol(\id_\Gamma)=\Vol(M)\,,
\eqn
where $\id_\Gamma\colon\Gamma\hookrightarrow\isomp(\HH^n)$
is the canonical injection.
\item (Rigidity)  There is equality
\bqn
\Vol(\rho)=\Vol(M)
\eqn
if and only if either:
\be
\item\label{eq:1} $n=2$ and $\rho$ is the holonomy representation of a (possibly infinite volume) complete hyperbolization 
of the smooth surface underlying $M$, \cite{Goldman_thesis, Burger_Iozzi_difff, Koziarz_Maubon}, or
\item\label{eq:2} $n=3$ and $\rho$ is conjugate to $\id_\Gamma$, \cite{Dunfield, Francaviglia_Klaff, Besson_Courtois_Gallot_comm, Bucher_Burger_Iozzi_Mostow}.
\ee
\item\label{eq:3} The volume function is continuous on $\rep(\Gamma,\isomp(\HH^n))$ (see Proposition~\ref{thm:cont} in Appendix~\ref{sec:continuity}).
\item\label{eq:4} If either $M$ is compact and $n\geq2$ \cite{Reznikov} (see also \cite{DLSW19})
or $M$ is finite volume and $n\geq4$ \cite{Kim_Kim_OnDeformation}, the volume is constant on connected components
of the representation variety.  As a consequence it takes only finitely many values.
\item\label{eq:5} If $M$ is a non-compact surface, the range of $\Vol$ coincides with the interval $[-\chi(M),\chi(M)]$,
where $\chi(M)$ is the Euler characteristic of $M$.
\item\label{eq:6} If $M$ is compact and $n$ is even, then
\bqn
\frac{2\Vol(\rho)}{\Vol(S^{n})}\in\ZZ\,,
\eqn
where here and in the sequel, $\Vol(S^{n})$ is the volume of the $n$-sphere $S^{n}$ of constant curvature 1.
\ee

Our main result is a generalization of the integrality property \eqref{eq:6}
to the case in which $M$ is not compact, and $n$ is even and $\geq4$. 
We remark that this is in sharp contrast with \eqref{eq:5}.

\begin{thm}\label{thm: main thm} 
Let $\Gamma < \mathrm{Isom}^+(\mathbb{H}^{2m})$ be a torsion-free lattice and 
let $\rho\colon \Gamma \rightarrow \mathrm{Isom}^+(\mathbb{H}^{2m})$ be a representation. 
Assume that $2m\geq4$.
\begin{enumerate}
\item If the manifold $M=\Gamma\setminus \mathbb{H}^{2m}$ has only toric cusps, then
\bqn
\frac{2 \Vol(\rho)}{\Vol(S^{2m})}\in \ZZ\,.
\eqn
\item In general 
\bqn
\frac{2 \Vol(\rho)}{\Vol(S^{2m})}\in \frac{1}{B_{2m-1}}\cdot \ZZ\,,
\eqn
where $B_{2m-1}$ is the Bieberbach number in dimension $2m-1$.
\end{enumerate}
\end{thm}
We recall that the Bieberbach number is the smallest integer $B_d$ such that any compact flat $d$-manifold
has a covering of degree $B_d$ that is a torus.
Such $d$-manifolds occur as connected components of the boundary of a compact core.
Recall in fact that, in the context of hyperbolic geometry a compact core $N$ of $M$ is a compact submanifold
that is obtained as the quotient by $\Gamma$ of the complement in $\HH^{d+1}$ of a $\Gamma$-invariant family
of pairwise disjoint open horoballs centered at the cusps.  

The strategy of the proof of Theorem~\ref{thm: main thm} builds on results in \cite{Burger_Iozzi_Wienhard_toledo},
where the authors studied the case in which $\dim M=2$ and established congruence relations for $\Vol(\rho)$.
In order to implement this strategy, we show that the continuous bounded class
\bqn
\omega_{2m}^{\mathrm{b}}\in\hcb^{2m}(\SO(2m,1)^\circ,\RR)\,,
\eqn
defined by the volume form on $\HH^{2m}$,
has a canonical representative
\bqn
\varepsilon_{2m}^{\mathrm{B,b}}\in\hb^{2m}(\SO(2m,1)^\circ,\ZZ)
\eqn
in the bounded Borel cohomology of $\SO(2m,1)^\circ$ that, under the change of coefficients $\ZZ\to\RR$, 
corresponds to $(-1)^m\frac{2}{\Vol(S^{2m})}\omega_{2m}^{\mathrm{b}}$.
For $2m=2$, $\varepsilon_{2}^{\mathrm{b}}$ coincides with the classical bounded Euler class as defined in \cite{Ghys}.
Then we establish a congruence relation modulo $\ZZ$ for $(-1)^m\frac{2}{\Vol(S^{2m})}\Vol(\rho)$ in terms of invariants attached
to the boundary components of $N$ that are now assumed to be $(2m-1)$-tori.
If $T_i$ is a component of $\partial N$ and $\rho_i\colon\ZZ^{2m-1}\to\SO(2m,1)^\circ$
is the restriction of $\rho$ to $\pi_1(T_i)\simeq\ZZ^{2m-1}$,
then the invariant attached to $T_i$ is
\bqn
\rho_i^\ast(\varepsilon_{2m}^{\mathrm{b}})\in\hb^{2m}(\ZZ^{2m-1},\ZZ)\simeq\RR/\ZZ\,.
\eqn
In the case in which $m=1$, $\rho_i^\ast(\varepsilon_2^{\mathrm{b}})$ 
coincides with the negative of the rotation number of $\rho_i^\ast(1)\in\SO(2,1)^\circ$
and we show in \S~\ref{sec:vanishing} that, if $m\geq2$, $\rho_i^\ast(\varepsilon_{2m}^{\mathrm{b}})$ always vanishes.

\begin{rem*}
In general $\Gamma$ has always a subgroup of finite index all whose cusps are toric.
However little is known about which collections of compact flat $(2m-1)$-manifolds $N$ are the components of the boundary of a compact core as above.
If $\dim M=4$, it is known that there are compact flat $3$-manifolds that are not diffeomorphic to $\partial N$, \cite[Corollary~1.4]{Long_Reid_00}, 
while on the positive side there are hyperbolic $4$-manifolds with $\partial N$ that is a $3$-torus, \cite{Kolpakov_Martelli}.
Which leads to the following:
\end{rem*}

\begin{question}  If $\Lambda$ is the fundamental group of a compact flat $(2m-1)$-manifold
and $\rho\colon\Lambda\to\SO(2m,1)^\circ$ is any representation,
does $\rho^\ast(\varepsilon_{2m}^{\mathrm{b}})$ vanish for $2m\geq4$?
\end{question}

Thus it is really in all odd dimensions that the nature of the values of $\Vol$ remain mysterious, though,
according to the results in \cite{Kim_Kim_OnDeformation}, for $n\geq4$ there are only finitely many possibilities.

We end this introduction by giving a result in dimension $3$.
In this case, the character variety of $\Gamma<\isomp(\HH^3)$ is smooth near $\id_\Gamma$,
and its complex dimension near $\id_\Gamma$ equals the number $h$ of cusps of $M$ \cite{Thurston_notes}.
As a consequence of the volume rigidity theorem and the continuity of $\Vol$, 
the image of $\Vol$ contains at least an interval $[\Vol(M)-\ep,\Vol(M)]$
for some $\ep>0$.
Special points in the character variety of $\Gamma$ come from Dehn fillings of $M$.
Let $M_\tau$ denote the compact manifold obtained from $M$ by Dehn surgery along 
a choice of $h$ simple closed loops $\tau=\{\tau_1,\dots,\tau_h\}$.
If the length of each geodesic loop $\tau_j$ is larger than $2\pi$,
$M_\tau$ admits a hyperbolic structure \cite{Thurston_notes} and an analytic formula for
$\Vol(M_\tau)$ depending on the length of the $\tau_j$ has been given in \cite{Neumann_Zagier}.

\begin{prop}\label{prop:intro}  Let $M_\tau$ be the compact $3$-manifold obtained by Dehn filling
from the hyperbolic $3$-manifold $M$.  If $\rho_\tau\colon \pi_1(M)\to\isomp(\HH^3)$
is the representation obtained from the composition of the quotient homomorphism
$\pi_1(M)\to\pi_1(M_\tau)$ 
with the holonomy representation of the hyperbolic structure on $M_\tau$, then
\bqn
\Vol(\rho_\tau)=\Vol(M_\tau)\,.
\eqn
\end{prop}

Thus with our cohomological definition, $\Vol(\rho)$ gives a continuous interpolation 
between the special values $\Vol(M_\tau)$.  A natural question here is whether $\Vol$ is real analytic.

\bigskip
The structure of the paper is as follows.  In \S~\ref{sec:0} we summarize the main facts about various  group cohomology
theories used in this paper.  In \S~\ref{sec:euler} we define a Borel cohomology class
$\varepsilon_{2m}\in\h_\mathrm{B}^{2m}(\SO(2m,1),\ZZ_\epsilon)$ with coefficients (see \S~\ref{sec:0})
and relate it to an explicit multiple of the class
$\omega_{2m}\in\h_\mathrm{B}^{2m}(\SO(2m,1),\RR_\epsilon)$ defined by the volume form on $\HH^{2m}$ (see \eqref{eq:vanEst+}).
In \S~\ref{sec:bddEuler}  we first show that $\varepsilon_{2m}$ has a unique bounded representative
$\varepsilon_{2m}^\mathrm{b}\in\h_{\mathrm{B,b}}^{2m}(\SO(2m,1),\ZZ_\epsilon)$ (Proposition~\ref{prop:bdd-euler});
in \S~\ref{sec:congruence} we proceed to define the volume $\Vol(\rho)$ of a representation
$\rho\colon\pi_1(M)\to\isomp(\HH^{2m})=\SO(2m,1)^\circ$
and use the bounded integral class $\varepsilon_{2m}^\mathrm{b}$ to establish a congruence relation for $\Vol(\rho)$ (Theorem~\ref{thm: vol modulo boundary}).
In \S~\ref{sec:vanishing}, which is the core of the paper, we show that the contributions from the various boundary components
of a compact core of $M$ to the congruence relation all vanish for $2m\geq4$.
In \S~\ref{sec:examples} we relate the volume of the representations of $\pi_1(M)$ obtained by Dehn surgery
to the volumes of the corresponding manifolds.  In the Appendix we prove the continuity of the map
$\rho\mapsto\Vol(\rho)$.

\section{Various cohomology theories}\label{sec:0}
We collect in this section cohomological results that will be used throughout the paper.

Given a locally compact second countable group $G$, we consider $\ZZ$, $\RR$ and $\RR/\ZZ$
as trivial modules.  If $\epsilon\colon G\to\{-1,+1\}$ is a continuous homomorphism,
we denote by $\ZZ_\epsilon$ and $\RR_\epsilon$ the corresponding coefficient $G$-modules,
where $g_\ast t=\epsilon(g) t$ for $t\in\ZZ, \RR$ and 
by $\RR_\epsilon/\ZZ_\epsilon$ the corresponding quotient module.

\medskip
If $A$ is any of the above $G$-modules, $\h_\mathrm{B}^\bullet (G,A)$ denotes the cohomology
of the complex of Borel measurable $A$-valued cochains on $G$.  
If $A=\RR,\RR_\epsilon$, we will also need the continuous cohomology $\hc^\bullet(G,A)$
with coefficients in $A$ and we will use that for $A=\RR,\RR_\epsilon$ the comparison map
\bq\label{eq:0.1}
\xymatrix{
\hc^\bullet(G,A)\ar[r]^-\simeq
&\h_\mathrm{B}^\bullet(G,A)
}
\eq
is an isomorphism \cite[Theorem~A]{Austin_Moore}.

\medskip
To compute the continuous cohomology of $G$ we use repeatedly that
if $G\times V\to V$ is a proper smooth action of a Lie group $G$ on a smooth manifold $V$,
there is an isomorphism
\bq\label{eq:0.2}
\hc^\bullet(G,\RR)\simeq\h^\bullet(\Omega^\bullet(V)^G)
\eq
with the cohomology of the complex $\Omega^\bullet(V)^G$ of $G$-invariant differential forms on $V$,
\cite[Theorem~6.1]{Hochschild_Mostow}.
If $V$ is a symmetric space $G/K$, then
\bq\label{eq:0.25}
\hc^\bullet(G,\RR)\simeq\Omega^\bullet(V)^G\,,
\eq
since every $G$-invariant differential form on $G/K$ is closed.

\medskip
Another result we will use is Wigner's isomorphism \cite{Wigner},
or rather a special case thereof \cite[Theorem~E]{Austin_Moore}:
namely if $A=\ZZ$ or $A=\ZZ_\epsilon$, there is  a natural isomorphism
\bq\label{eq:0.3}
\h_\mathrm{B}^\bullet(G,A)\simeq\h_\mathrm{sing}^\bullet(BG,A)\,,
\eq
where $BG$ is the classifying space of $G$ and $\h_\mathrm{sing}^\bullet$ refers to singular cohomology.

\medskip
A vanishing theorem that is often used states that if $L$ is compact, then
\bq\label{eq:0.4}
\hc^k(L,\RR)=0\quad\text{ and }\quad\hc^k(L,\RR_\epsilon)=0
\eq
for all $k\geq1$.

\medskip
Turning to bounded cohomology, $\h_\mathrm{B,b}^\bullet(G,A)$ denotes the cohomology of bounded $A$-valued
Borel cochains on $G$.  An important point is that if $A=\RR$ or $A=\RR_\epsilon$,
the comparison map from continuous bounded to Borel bounded cohomology induces an isomorphism
\bq\label{eq:0.5}
\xymatrix{
\hcb^\bullet(G,A)\ar[r]^-\simeq
&\h_\mathrm{B,b}^\bullet (G,A)
}
\eq
as can be readily verified using the regularization operators defined in \cite[\S~4]{Blanc}.

\medskip
Analogously to the vanishing of the continuous cohomology for compact groups,
if $P$ is amenable, then
\bq\label{eq:06}
\hcb^k(P,\RR)\cong \h^k_\mathrm{B,b}(P,\RR)=0
\eq
for all $k\geq1$.

\medskip
Consider now the two short exact sequence of coefficients

\bq\label{eq:0.8-}
\xymatrix@1{
0\ar[r]
&\ZZ\,\,\ar@{^{(}->}[r]
&\RR\ar@{->>}[r]
&\RR/\ZZ\ar[r]
&0\, ,
}
\eq
\bqn
\xymatrix@1{
0\ar[r]
&\Zep\ar@{^{(}->}[r]
&\Rep\ar@{->>}[r]
&\Rep/\Zep\ar[r]
&0\, .
}
\eqn
Using that $\RR\to \RR/\ZZ $ admits a bounded Borel section, 
one obtains readily, both for the trivial and nontrivial modules, long exact sequences in Borel and bounded Borel cohomology
with commutative squares coming from the comparison maps $c_{\ZZ}$ and $\c_{\RR}$ between these two cohomology theories:

 {\small
 \bq \label{eqn:0.8}
 \xymatrix{
 \dots\ar[r]
&\h_\mathrm{B}^{2m-1}(G,\RR/\ZZ)\ar[r]^-{\delta^{\mathrm{b}}}\ar@{=}[d]
&\h_\mathrm{B,b}^{2m}(G,\ZZ)\ar[r]\ar[d]^{c_{\ZZ}}
&\h_\mathrm{B,b}^{2m}(G,\RR)\ar[r]\ar[d]^{c_{\RR}}
&\h_\mathrm{B}^{2m}(G,\RR/\ZZ)\ar[r]\ar@{=}[d]
&\dots\\
\dots\ar[r]
&\h_\mathrm{B}^{2m-1}(G,\RR/\ZZ)\ar[r]_-\delta
&\h_\mathrm{B}^{2m}(G,\ZZ)\ar[r]
&\h_\mathrm{B}^{2m}(G,\RR)\ar[r]
&\h_\mathrm{B}^{2m}(G,\RR/\ZZ)\ar[r]
&\dots\,,
 }
 \eq}
 and
 {\small
 \bq \label{eq:0.7}
 \xymatrix{
 \dots\ar[r]
&\h_\mathrm{B}^{2m-1}(G,\Rep/\Zep)\ar[r]^-{\delta^{\mathrm{b}}}\ar@{=}[d]
&\h_\mathrm{B,b}^{2m}(G,\Zep)\ar[r]\ar[d]^{c_{\ZZ}}
&\h_\mathrm{B,b}^{2m}(G,\Rep)\ar[r]\ar[d]^{c_{\RR}}
&\h_\mathrm{B}^{2m}(G,\Rep/\Zep)\ar[r]\ar@{=}[d]
&\dots\\
\dots\ar[r]
&\h_\mathrm{B}^{2m-1}(G,\Rep/\Zep)\ar[r]_-\delta
&\h_\mathrm{B}^{2m}(G,\Zep)\ar[r]
&\h_\mathrm{B}^{2m}(G,\Rep)\ar[r]
&\h_\mathrm{B}^{2m}(G,\Rep/\Zep)\ar[r]
&\dots\,,
 }
 \eq}
 where $\delta$ and $\delta^\mathrm{b}$ are the connecting homomorphisms.

\section{Proportionality between volume and Euler class}\label{sec:euler}\label{sec:newsectionProp}

Let $\isom(\HH^n)$ be the full group of isometries of real hyperbolic spaces $\HH^n$,
let $\ep\colon \isom(\HH^n)\to\{-1,1\}$ denote the homomorphism with kernel $\isomp(\HH^n)$
and let $\Zep\subset\Rep$ the corresponding modules.  
Using \eqref{eq:0.1}, \cite[Proposition~2.1]{Bucher_Burger_Iozzi_Mostow} and \eqref{eq:0.25},
we have isomorphisms
\bq\label{eq:vanEst+}
\h_\mathrm{B}^n(\isom(\HH^n),\Rep)
\simeq \hc^n(\isom(\HH^n),\Rep)
\simeq \hc^n(\isom(\HH^n)^\circ,\RR)
\simeq\Omega(\HH^n)^{\isom(\HH^n)^\circ}	
\eq
and denote by $\omega_n\in\h_\mathrm{B}^n(\isom(\HH^n),\Rep)$ the generator corresponding to the volume form on $\HH^n$.

If $n=2m$ is even, we can identify $\isom(\HH^{2m})$ with $\SO(2m,1)$.  
The diagram of injections
\bqn\label{equ:inclusions}
\xymatrix{
&\mathrm{O}(2m)\ar@{^{(}->}[dr]\ar@{_{(}->}[dl]& \\ 
\mathrm{GL}(2m,\RR) &  & \mathrm{SO}(2m,1)}
\eqn
realizes $\O(2m)$ as a maximal compact subgroup of both $\SO(2m,1)$ and $\GL(2m,\RR)$
and induces homotopy equivalences
\bqn
B\GL(2m,\RR)\simeq B\O(2m)\simeq B\SO(2m,1)\,.
\eqn
The homomorphism
\bqn
\GL(2m,\RR)\to\{-1,1\}
\eqn
associating the sign of the determinant, coincides on $\O(2m)$ with the restriction of $\epsilon$ 
and will be denoted by $\epsilon$.  Thus we obtain isomorphisms
\bq\label{eq:2.1}
\h_\mathrm{sing}^{2m}(B\GL(2m,\RR),\Zep)
\cong \h_\mathrm{sing}^{2m}(B\O(2m),\Zep)
\cong \h_\mathrm{sing}^{2m}(B\SO(2m,1),\Zep)\,.
\eq


The (universal) Euler class $\vem^\mathrm{univ}\in \h^{2m}_\mathrm{sing}(B\mathrm{GL}^+(2m,\RR),\ZZ)$ 
is a singular class in the integral cohomology of the classifying space $B\mathrm{GL}^+(2m,\RR)$ of oriented $\RR^{2m}$-vector bundles
 (see \cite[\S~9]{Milnor_Stasheff}). It is the obstruction to the existence of a nowhere vanishing section. 
 As it changes sign when the orientation is reversed, it extends to a class $\vem^\mathrm{univ}\in \h^{2m}_\mathrm{sing}(B\mathrm{GL}(2m,\RR),\Zep)$. 
 Furthermore, if $M$ is a closed oriented $2m$-dimensional manifold, 
 its tangent bundle is classified by a (unique up to homotopy) classifying map 
 \bqn
 f\colon M\rightarrow B \mathrm{GL}^+(2m,\RR) \hookrightarrow B\mathrm{GL}(2m,\RR)\simeq B\mathrm{O}(2m)
 \eqn 
 inducing a map 
 \bqn
 f^*\colon \h^{2m}_\mathrm{sing}(B\mathrm{GL}(2m,\RR),\Zep)\cong\h^{2m}_\mathrm{sing}(B\mathrm{O}(2m)\RR,\Zep) \rightarrow \h^{2m}_\mathrm{sing}(M,\ZZ)
 \eqn 
 and thus
\bq \label{equ:EulerChar}
\chi(M)=\langle f^*(\vem^\mathrm{univ}),[M]\rangle\,.
\eq
Now we use Wigner's isomorphism \eqref{eq:0.3}
\bq\label{eq:2.3}
\h_\mathrm{sing}^{2m}(B\SO(2m,1),\Zep)\cong \h_\mathrm{B}^{2m}(\SO(2m,1),\Zep)
\eq
and call Euler class the group cohomology class
\bqn
\varepsilon_{2m}\in\h_\mathrm{B}^{2m}(\SO(2m,1),\Zep)
\eqn
corresponding to $\varepsilon_{2m}^\mathrm{univ}$ under the composition of the isomorphisms
in \eqref{eq:2.1} and \eqref{eq:2.3}.


Since $\h^{2m}_\mathrm{B}(\mathrm{SO}(2m,1),\Rep)$ is one dimensional by \eqref{eq:2.1}, 
the image of $\vem$ under the change of coefficients $\Zep\hookrightarrow \Rep$ is a multiple of the volume class $\omega_{2m}$. 
We show now that this multiple is given by 
\bq \label{eqn:Hirzebruch}
\ba
\hhc^{2m}(\mathrm{SO}(2m,1),\Zep)&\longrightarrow\h_\mathrm{B}^{2m}(\mathrm{SO}(2m,1),\Rep)\\
\vem\quad\qquad&\longmapsto(-1)^m\frac{2}{\Vol(S^{2m})}\omega_{2m}.
\ea
\eq
Indeed, let $M$ be a closed oriented hyperbolic manifold. 
The lattice embedding $i\colon \pi_1(M)\hookrightarrow \mathrm{SO}(2m,1)$ induces classifying maps
\bq \label{eq:ugly arrow}
\xymatrix{ 
M=K(\pi_1(M),1) \ar[rr]^-{Bi} \ar@/_2.5pc/[rr]^{\overline{Bi}}
&&B\mathrm{SO}(2m,1)\simeq B\mathrm{O}(2m)\simeq 
B\mathrm{GL}(2m,\RR) .  
}
\eq
Set $\Gamma=i(\pi_1(M))$. The orthonormal frame bundle over $M$ is naturally identified with
\bqn
\Gamma\setminus \mathrm{SO}(2m,1)\, .
\eqn
Extending the principal group structure from $\mathrm{O}(2m)$ to $\mathrm{SO}(2m,1)$ gives the principal $\mathrm{SO}(2m,1)$-bundle
\bqn 
(\Gamma \setminus \mathrm{SO}(2m,1))\times_{\mathrm{O}(2m)} \mathrm{SO}(2m,1)\, .
\eqn
The latter is isomorphic to 
\bqn
(\mathrm{SO}(2m,1)/ \mathrm{O}(2m))\times_\Gamma \mathrm{SO}(2m,1)\, , 
\eqn
which is the flat principal $\mathrm{SO}(2m,1)$-bundle associated to the lattice embedding $i\colon \pi_1(M)\hookrightarrow \mathrm{SO}(2m,1)$. It follows that the composition from (\ref{eq:ugly arrow})
\bqn
\overline{Bi}:M\longrightarrow B\mathrm{O}(2m)\simeq B\mathrm{GL}(2m,\RR)
\eqn
classifies the tangent bundle $TM$. As a consequence, 
\bqn
\chi(M)=\langle \overline{Bi}^*(\vem^\mathrm{univ}),[M]\rangle\,.
\eqn
Thus, by the naturality of Wigner's isomorphism, the following diagram commutes:
\bqn
\xymatrix{ 
  & \h^{2m}_\mathrm{sing}(B\mathrm{GL}(2m,\RR),\Zep)  \ar[ld]_-{\overline{Bi}}\ar@{=}[d]\\
  \h^{2m}_\mathrm{sing}(M,\ZZ)   &    \h^{2m}_\mathrm{sing}(B\mathrm{SO}(2m,1),\Zep)  \ar[l]^-{Bi}\\
   \h^{2m}(\pi_1(M),\ZZ)  \ar[u]^\cong &    \h^{2m}_\mathrm{B}(\mathrm{SO}(2m,1),\Zep).  \ar@{=}[u]\ar[l]^{i^*}
 }
\eqn
We deduce that 
\bqn
\chi(M)=\langle i^*\vem,[M]\rangle\,.
\eqn
Moreover, the hyperbolic volume of $M$ is obviously given by 
\bqn
\mathrm{vol}(M)=\langle i^*\omega_{2m},[M]\rangle\,.
\eqn

The Chern--Gauss--Bonnet Theorem \cite[Chapter~13, Theorem~26]{Spivak} implies that 
if $M$ is a hyperbolic $(2m)$-dimensional manifolds
then
\bqn
\chi(M)=(-1)^m\frac{2}{\Vol(S^{2m})}\mathrm{vol}(M)\,,
\eqn
where the proportionality constant is up to the sign $(-1)^m$ the corresponding quotient between the Euler characteristic 
and the volume of the compact dual of hyperbolic space, namely the $(2m)$-sphere of constant curvature $+1$. 
Finally, since the lattice embedding induces an injection 
\bqn
i^*\colon \h_\mathrm{B}^{2m}(\mathrm{SO}(2m,1),\Rep)\longrightarrow  \h^{2m}(\pi_1(M),\RR)\,,
\eqn
we obatin the value of the proportionality constant in (\ref{eqn:Hirzebruch}).

\begin{rem}\label{rem:3.1} Recall that the orientation cocycle 
\bq \mathrm{Or}\colon (S^1)^3\longrightarrow \{+1,0,-1\}
\eq
assigns the value $\pm 1$ to distinct triple of points according to their orientation, and $0$ otherwise. 
Identifying $S^1$ with $\partial \HH^2$ we obtain by evaluation a Borel cocycle and a cohomology class
\bqn
[\mathrm{Or}]\in \h_\mathrm{B}^2(\isom(\HH^2),\Rep)\,.
\eqn
Since the area of ideal hyperbolic triangles in $\HH^2$ is $\pm \pi$ depending on the orientation, 
this cocycle represents $(1/\pi)\omega_2=-2\varepsilon_2$, 
where the last equality comes from (\ref{eqn:Hirzebruch}). 
Thus the Euler class $\varepsilon_2$ is represented by $-\frac12\mathrm{Or}$,
\cite[Lemma~2.1]{Iozzi_ern} and \cite[Proposition~8.4]{Bucher_Monod}.
\end{rem}

\section{The bounded Euler class}\label{sec:bddEuler} 

In this section we recall that the volume class of hyperbolic $n$-space admits a unique bounded representative 
and establish for $n$ even the analogous result for the Euler class, 
which leads to the definition of the bounded Euler class  in Proposition~\ref{prop:bdd-euler}. 
In \S~\ref{sec:congruence} we define the volume of a representation 
and use the existence of the bounded Euler class in even dimensions 
to prove in Theorem \ref{thm: vol modulo boundary} a congruence relation for the volume of a representation. 

\subsection{Bounded volume and Euler classes}


\begin{lemma}\label{lem:comp}  Let $G:=\isom(\HH^n)$. In the commuting diagram
 {
 \bqn
 \xymatrix{
\hhcb^{n}(G,\Zep)\ar[r]\ar[d]_{c_\ZZ}
&\hhcb^{n}(G,\Rep)\ar[d]^{c_\RR}
\\
\hhc^{n}(G,\Zep)\ar[r]
&\hhc^{n}(G,\Rep)
 }
 \eqn}
the vertical maps are isomorphisms\,.
\end{lemma}

\begin{proof} 
The fact that $c_\RR$ is an isomorphism follows from \cite[Proposition~2.1]{Bucher_Burger_Iozzi_Mostow} 
and the identifications between the (bounded) Borel and the (bounded) continuous cohomology.

To show that $c_\ZZ$ is an isomorphism, we will do diagram chases in  (\ref{eq:0.7}).

\medskip
\noindent
{\em Surjectivity of $c_\ZZ$:}   Let $\alpha\in \hhc^{n}(G,\Zep)$. Denote by $\alpha_\RR \in \hhc^{n}(G,\Rep)$ 
its image under the change of coefficients and $\alpha_\RR^{\mathrm b}=(c_\RR)^{-1}(\alpha_\RR)\in \hhcb^{n}(G,\Rep)$. 
By exactness of the lines in (\ref{eq:0.7}), the image of $\alpha_\RR$ in $\hhc^{n}(G,\Rep/\Zep)$ vanishes. 
And thus the same holds for the image of $\alpha_\RR^{\mathrm b}$. 
Thus there is $\beta\in \h_\mathrm{B}^{n}(G,\Zep)$ with image $\alpha_\RR^{\mathrm b}$. But $\c_\ZZ(\beta)-\alpha$ goes to zero in 
 $\hhc^{n}(G,\Rep)$, hence $\c_\ZZ(\beta)-\alpha=\delta(\eta)$ for some 
 $\eta\in\hhc^{n-1}(G,\Rep/\Zep)$.  Thus $c_\ZZ(\beta+\delta^{\mathrm{b}}(\eta))=\alpha$.

\medskip
\noindent
{\em Injectivity of  $c_\ZZ$:}  Observe that $\h_\mathrm{B}^{n-1}(G,\Rep)=\hc^{n-1}(G,\Rep)=0$.  Indeed this group injects into
$\hc^{n-1}(G^\circ,\RR)$ by restriction, and the latter vanishes, taking into account \eqref{eq:0.2},
since there are no $\SO(n)$-invariant $(n-1)$-forms on $\RR^n$ 
and hence no $G^\circ$-invariant differential $(n-1)$-forms on $\HH^{n}$.

Let now $\alpha\in\hhcb^{n}(G,\Zep)$ with $c_\ZZ(\alpha)=0$.  
Since $c_\RR$ is an isomorphism,
we have that the image $\alpha_\RR\in\h_\mathrm{B,b}^{n}(G,\Rep)$ vanishes.
Hence there is $\beta\in\h_\mathrm{B}^{n-1}(G,\Rep/\Zep)$ with $\delta^{\mathrm{b}}(\beta)=\alpha$.
But then $\delta(\beta)=c_\ZZ(\alpha)=0$.  
By exactness, this implies that $\beta$ is in the image of $\h_\mathrm{B}^{n-1}(G,\Rep)$. Since $\h_\mathrm{B}^{n-1}(G,\Rep)=0$,
then $\beta=0$, which finally implies  that $\alpha=0$.
\end{proof}

As a consequence of the fact that $c_\RR$ is an isomorphism and that the volume of geodesic simplices in hyperbolic $n$-space is bounded, 
the volume cocycle defines also a bounded Borel class 
\bqn
\omega_n^{\mathrm{b}}\in\h_\mathrm{B,b}^n(G,\Rep)
\eqn
corresponding to the volume class $\omega_n$.

As an immediate consequence of Lemma \ref{lem:comp} and the correspondence (\ref{eqn:Hirzebruch}) we obtain: 

\begin{prop}\label{prop:bdd-euler}  Let $G=\isom(\HH^{2m})$.  
The Euler class $\vem\in\hhc^{2m}(G,\Zep)$
has a bounded representative $\vemb\in\hhcb^{2m}(G,\Zep)$ that 
has the following properties:
\be
\item it is unique, and
\item under the change of coefficients $\Zep\to\Rep$ it corresponds to 
\bqn
(-1)^m\frac{2}{\Vol(S^{2m})}\omb\in\hhcb^{2m}(G,\Rep)\,.
\eqn
\ee
\end{prop}

\begin{rem}
\be
\item  With a slight abuse of notation we denote equally by $\vemb\in\hhcb^{2m}(\O(2m),\Zep)$ and by $\vemb\in\hhcb^{2m}(\SO(2m),\ZZ)$
the restriction of $\vemb$ respectively to $\O(2m)$ and to $\SO(2m)$.
\item If $m=1$ and with the usual slight abuse of notation, the restriction $\varepsilon_2^\mathrm{b}\in\hhcb^2(\isom(\HH^{2})^\circ,\ZZ)$ 
is the usual bounded Euler class, 
where $\isom(\HH^{2})^\circ$ is considered as a group of orientation preserving homeomorphisms of the circle.
\ee
\end{rem}

\subsection{Definition of volume and congruence relations}\label{sec:congruence}
Let $M$ be a complete finite volume hyperbolic $n$-dimensional manifold and 
let $\rho\colon \pi_1(M)\to \isomp(\HH^n)$ be a homomorphism.
Given a compact core $N$ of $M$ we consider $\rho$ as a representation of $\pi_1(N)$ and use
the pullback via $\rho$ in bounded cohomology together with the isomorphism
\bqn
\hb^n(\pi_1(N),\RR)\cong\hb^n(N,\RR)
\eqn
to obtain a bounded singular class in $\hb^n(N,\RR)$, denoted $\rho^*(\omega_n^\mathrm{b})$ by abuse of notation.
Using the isometric isomorphism
\bq\label{eq:j}
j\colon \hb^{2m}(N,\partial N,\RR)\longrightarrow\hb^{2m}(N,\RR)
\eq
\cite[Theorem~1.2]{BBFIPP},
the volume of $\rho$ is defined by 
\bqn
\Vol(\rho):=\langle j^{-1}(\rho^*(\omb)),[N,\partial N]\rangle\,.
\eqn
If $n=2m$, by the same abuse of notation, and by considering again the pullback in bounded
cohomology via $\rho$, this time together with the isomorphism
\bqn
\hb^{2m}(\pi_1(N),\ZZ)\cong
\hb^{2m}(N,\ZZ)\,,
\eqn
we obtain a class $\rho^*(\vemb)\in\hb^{2m}(N,\ZZ)$, which is a bounded singular integral class.
Thus, denoting $\delta^\mathrm{b}$ the connecting homomorphism in the long exact sequence 
in bounded singular cohomology
\bq\label{eq:delta}
\delta^\mathrm{b}\colon  \h^{2m-1}(\partial N,\RR/\ZZ)\longrightarrow\hb^{2m}(\partial N,\ZZ)
\eq
(which is in fact an isomorphism), we have:

\begin{thm} \label{thm: vol modulo boundary} 
Let $M$ be a complete hyperbolic manifold of finite volume and even dimension $n=2m$
and $N$ a compact core of $M$. 
If $\rho\colon \pi_1(M)\rightarrow \isomp(\HH^{2m})$, then
\bqn
(-1)^m\frac{2}{\Vol(S^{2m})}\cdot \Vol(\rho)\equiv-\langle (\delta^\mathrm{b})^{-1} \rho^*(\vemb)|_{\partial N} , [\partial N] \rangle\,\mathrm{mod}\, \ZZ\,.
\eqn
\end{thm}

\begin{proof} In fact, \eqref{eq:j} and \eqref{eq:delta} are part of the following diagram
\bqn
\xymatrix{ 
   \h^{2m-1}(\partial N,\RR/\ZZ)\ar[r]^-{\delta^\mathrm{b}}
&\hb^{2m}(\partial N,\ZZ)\ar[r]
&\hb^{2m}(\partial N,\RR)=0\\ 
&\hb^{2m}( N,\ZZ)\ar[r]\ar[u]
&\hb^{2m}(N,\RR)\ar[u]\\
&&\hb^{2m}(N,\partial N,\RR)\ar[u]_j\,,
}
\eqn
where the rows are obtained from the long exact sequence 
\bqn
\xymatrix@1{
\dots\ar[r]
&\hb^{n-1}(X,\RR/\ZZ)\ar[r]
&\hb^n(X,\ZZ)\ar[r]
&\hb^n(X,\RR)\ar[r]
&\hb^n(X,\RR/\ZZ)\ar[r]
&\dots\,,
}
\eqn
with $X=\partial N$ and $X=N$,
induced by the change of coefficients in 
\eqref{eq:0.8-}
and from the fact that $\hb^\bullet(\partial N,\RR)=0$ since $\pi_1(\partial N)$ is amenable;
the columns on the other hand follow from the long exact sequence in relative bounded
cohomology associated to the inclusion of pairs $(N,\varnothing)\hookrightarrow(N,\partial N)$
(see \cite[\S~2.2]{Burger_Iozzi_Wienhard_toledo}).

Let $z$ be a $\ZZ$-valued singular bounded cocycle representing $\rho^*(\vemb)\in\hb^{2m}(N,\ZZ)$. 
Restricting $z$ to the boundary $\partial N$ we obtain a $\ZZ$-valued singular bounded cocycle $z|_{\partial N}$,
which we know is a coboundary when considered as a $\RR$-valued cocycle since $\hb^{2m}(\partial N,\RR)=0$. 
Thus there must exist a bounded $\RR$-valued singular $(2m-1)$-cochain $b$ on $\partial N$ such that
\bqn
(z|_{\partial N})_\RR=d b\,,
\eqn
where $d$ is the differential operator on (bounded $\RR$-valued) singular cochains.

On the one hand, we note that since $d b$ is $\ZZ$-valued, 
the cochain $b\,\mathrm{mod}\, \ZZ$ is a $\RR/\ZZ$-valued $(2m-1)$-cocycle on $\partial N$ 
whose cohomology class is mapped to 
\bqn
[\rho^*(\vemb)|_{\partial N}]=[z|_{\partial N}]=\delta^{\mathrm{b}}([b\,\mathrm{mod}\,  \ZZ])
\eqn
by the connecting homomorphism $\delta^{\mathrm{b}}$ in \eqref{eq:delta}.

On the other hand, define a bounded $\RR$-valued singular $(2m-1)$-cochain $\overline{b}$ on $N$ 
by extending $b$ to $N$,
\bqn
\overline{b}(\sigma):= \left\{   \begin{array}{ll}
b(\sigma) & \mathrm{if \ } \sigma\subset \partial N, \\
0 & \mathrm{otherwise.}
\end{array} \right.
\eqn
Then $[z_\RR-d \overline{b}]=[z_\RR]\in\hb^{2m}(N,\RR)$, and since $z_\RR-d\overline{b}$ 
vanishes on $\partial N$ we have actually constructed a cocycle representing 
the relative bounded class $j^{-1}((\rho^*(\vemb))_\RR)\in \hb^{2m}(N,\partial N,\RR)$.

It remains to evaluate $j^{-1}((\rho^*(\vemb))_\RR)$ on $[N,\partial N]$ by using this specific cocycle. 
Let $t$ be a singular chain representing the relative fundamental class $[N,\partial N]$ over $\ZZ$.  
In particular, $\partial t$ is a cycle representing the fundamental class $[\partial N]$. Then we obtain
\begin{eqnarray*}
\langle j^{-1}((\rho^*(\vemb))_\RR), [N,\partial N]\rangle \,\mathrm{mod}\,  \ZZ&=&\langle z_\RR-d \overline{b}, t\rangle \,\mathrm{mod}\,  \ZZ\\
&=& \underbrace{\langle z_\RR, t\rangle}_{\in \ZZ} -\langle \overline{b}, \partial t \rangle  \,\mathrm{mod}\,  \ZZ\\
&=& -\langle b\,\mathrm{mod}\,  \ZZ, [\partial N]\rangle \\
&=&-\langle (\delta^{\mathrm{b}})^{-1}(\rho^*(\varepsilon_{2m})|_{\partial N}), [\partial N]\rangle. 
\end{eqnarray*}


\end{proof}

\section{Vanishing of the bounded Euler class on tori and the proof of Theorem \ref{thm: main thm}}\label{sec:vanishing}
The goal of this section is to prove Theorem \ref{thm: main thm}. 
From Theorem \ref{thm: vol modulo boundary}, we know that $(-1)^m2/\mathrm{Vol}(S^{2m})$ times the volume of a representation is determined, $\,\mathrm{mod}\,  \ZZ$, 
by the restriction to the cusps of the pullback of the bounded Euler class. 
The main result of this section will be to prove, in Theorem \ref{thm:vanishing},  
that the pullback of the bounded Euler class by any representation $ \rho\colon \ZZ^{2m-1}\to\SO(2m,1)^\circ$ is identically zero for $m\geq 2$. 
We conclude the section by showing how Theorems~\ref{thm: vol modulo boundary} and \ref{thm:vanishing} imply Theorem~\ref{thm: main thm}.

\medskip
The bounded class $\vemb\in\hhcb^{2m}(\SO(2m,1),\Zep)$ defined in the previous section restricts
to a class on $\SO(2m,1)^\circ$ with trivial $\ZZ$-coefficients.  When $m=1$ and 
\bqn
\rho\colon \ZZ\to\SO(2,1)^\circ
\eqn
is a homomorphism, then $\rho^*(\ve_2^{\mathrm{b}})\in\hb^2(\ZZ,\ZZ)\cong\RR/\ZZ$
is the negative of the rotation number of $\rho(1)$, \cite{Ghys}.  
In contrast to this, in higher dimension we have the following:

\begin{thm}\label{thm:vanishing}  Let $m\geq2$ and let $\rho\colon \ZZ^{2m-1}\to\SO(2m,1)^\circ$
be a homomorphism.  Then $\rho^*(\vemb)$ vanishes in $\hb^{2m}(\ZZ^{2m-1},\ZZ)$.
\end{thm}

Before proving the theorem we need some information about Abelian subgroups of $\SO(2m,1)^\circ$.  
To fix the notation, recall that 
\bqn
\SO(2m,1):=\bigg\{A\in\GL(2m+1,\RR):\,\det A=1,\,A\text{ preserves }q(x):=\sum_{i=1}^{2m}x_i^2-x_{2m+1}^2\bigg\}\,.
\eqn
Then the maximal compact subgroup $K<\SO(2m,1)$ is the image of $\O(2m)$ under the homomorphism
\bqn
\ba
\O(2m)&\longrightarrow\,\,\SO(2m,1)\\
A\quad&\longmapsto\begin{pmatrix}A&0\\0&\det A\end{pmatrix}\,,
\ea
\eqn
and the image $K^\circ$ of $\SO(2m)$ is the maximal compact subgroup of $\SO(2m,1)^\circ$.
If $T$ is the image of 
\bqn
\ba
\O(2)^m\quad&\longrightarrow\qquad\qquad\SO(2m,1)\\
(A_1,\dots,A_m)&\longmapsto
\begin{pmatrix}
A_1&           &       &\\
       &\ddots&       &\\
       &           &A_m&\\
       &           &       &\prod_{i=1}^m\det A_i
       \end{pmatrix}\,,
\ea
\eqn
we define
\bqn
T_0:=T\cap K^\circ\,.
\eqn

Since $\SO(2m,1)^\circ$ preserves each connected component of the two-sheeted hyperboloid
\bqn
x_1^2+\dots+x_{2m}^2-x_{2m+1}^2=-1\,,
\eqn
the parabolic subgroup $P=\Stab_{\SO(2m,1)^\circ}(\RR(e_1-e_{2m+1}))$
admits the decomposition $P=MAN$, where
\bqn
\ba
M:&=\left\{m(U):=\begin{pmatrix}1 &  &  \\ & U & \\ &  &  1\end{pmatrix}:\,U\in\SO(2m-1)\right\}\\
A:&=\left\{
a(t):=
\begin{pmatrix}
\cosh t&0&\sinh t\\
0&\id&0\\
\sinh t&0&\cosh t
\end{pmatrix}:\,
t\in\RR
\right\}
\ea
\eqn
and 
\bqn
N:=
\left\{
n(x):=
\begin{pmatrix}
1-\frac{\|x\|^2}{2}&-x&-\frac{\|x\|^2}{2}\\
{}^tx&I&{}^tx\\
\frac{\|x\|^2}{2}&x&1+\frac{\|x\|^2}{2}
\end{pmatrix}:\,x\in\RR^{2m-1}
\right\}\,.\eqn  

\medskip
We can now outline the proof of Theorem~\ref{thm:vanishing}.  
According to Lemma~\ref{lem:tricot} below, there are up to conjugation two cases to consider for 
$\rho:\ZZ^{2m-1}\to\SO(2m,1)^\circ$:
\begin{enumerate}
\item  $\rho$ takes values in $P$:  then, building on Lemma~\ref{lem:cohom of parabolic},
Lemma~\ref{lem:vanish} shows that $\varepsilon_{2m}^\mathrm{b}|_P\in\h_{\mathrm{B,b}}^{2m}(P,\ZZ)$
vanishes and this implies Theorem~\ref{thm:vanishing} in this case.
\item $\rho$ takes values in $T_0$:  this case splits into two subcases, namely $\rho(\ZZ^{2m-1})\not\subset T^\circ$,
which is dealt with in Lemma~\ref{lem:vanish2}, and $\rho(\ZZ^{2m-1})\subset T^\circ$.
In the latter case we represent $\varepsilon_{2m}^\mathrm{b}|_{\SO(2)^{2m}}$ as the image
under the connecting homomorphism of an explicit class in $\h_\mathrm{B}^{2m-1}(\SO(2)^{2m},\RR/\ZZ)$
whose pullback via $\rho$ we evaluate in Lemma~\ref{lem:rot=0}
on an explicit cycle representing the fundamental class $[\ZZ^{2m-1}]$ of $\h_{2m-1}(\ZZ^{2m-1},\ZZ)$
constructed in Lemma~\ref{lem: representation Zn}.
\end{enumerate}

\begin{lemma}\label{lem:tricot} Let $B<\SO(2m,1)^\circ$ be an Abelian group.
Then up to conjugation one of the following holds:
\be
\item $B<P$;
\item $B<T_0$.
\ee
\end{lemma}

\begin{proof}  Since $B$ is Abelian, it is elementary \cite[Lemma~1, \S~5.5]{Ratcliffe},
that is, it is either of elliptic type or of parabolic type or of hyperbolic type.
\be
\item[(i)] If $B$ is of parabolic type, it fixes a point in $\partial\HH^{2m}$  \cite[Theorem~5.5.1]{Ratcliffe}
and thus it can be conjugated into $P$.
\item[(ii)] If $B$ is of elliptic type it fixes a point in $\HH^{2m}$  \cite[Theorem~5.5.3]{Ratcliffe}
and hence it can be conjugated into $K^\circ\cong\SO(2m)$.
Since it is Abelian, it can be simultaneously reduced to a diagonal $2\times 2$ bloc form and hence
can be conjugated into $T_0$.
\item[(iii)] If $B$ is of hyperbolic type, every union of finite orbits in $\overline{\HH^{2m}}$ consists of two points,  \cite[Theorem~5.5.6]{Ratcliffe}.
Thus there is a geodesic $g\in\HH^{2m}$ that is left setwise invariant by $B$.  
If $\{g_-,g_+\}$ are its endpoints, then either $Bg_+=g_+$ and we are in case (i) above, or there is $b\in B$ with $bg_+=g_-$.  
But then $B$ fixes the unique $b$-fixed point $g_0\in g$, hence we are in case (ii).
\ee
\end{proof}

\begin{rem}  Since Theorems~5.5.1, 5.5.3 and 5.5.6 in  \cite{Ratcliffe} are characterizations respectively of elliptic,
parabolic and hyperbolic groups of isometries, 
it follows from Lemma~\ref{lem:tricot} that an Abelian group $B<\SO(2m,1)^\circ$ cannot be hyperbolic.
\end{rem}

\begin{lemma}\label{lem:cohom of parabolic}  Let $P<\SO(2m,1)^\circ$ be the minimal parabolic subgroup
as above.  Then:
\bqn
\h_\mathrm{B}^k(P,\RR)=0
\eqn
for  $k=2m-1$ and $k=2m$.
\end{lemma}

\begin{proof} 
%
%
Combining \eqref{eq:0.1} and \eqref{eq:0.2}, we show the vanishing assertion by showing the vanishing of the cohomology of $P$-invariant
differential forms on $M\backslash P$ in degrees $2m$ and $2m-1$.  
This follows if we establish the following:
\begin{claim}  $\Omega^{2m-1}(M\backslash P)^P=\RR\alpha$,  where $\alpha$ is an explicit $(2m-1)$ form with $d\alpha\neq0$.
\end{claim}

Let us first see how the claim implies the vanishing.  It follows from the claim that 
\bqn
d\colon\Omega^{2m-1}(M\backslash P)^P\to \Omega^{2m}(M\backslash P)^P
\eqn
is injective, which implies the vanishing in degree $2m-1$.  Since $P$ acts transitively on $M\backslash P$
and $\dim(M\backslash P)=2m$, then $\dim \Omega^{2m}(M\backslash P)^P\leq1$.
But then $d\alpha\neq0$ implies that $\dim\Omega^{2m}(M\backslash P)^P=1$,
the image of $d$ is $\Omega^{2m}(M\backslash P)^P$ which implies the vanishing in degree $2m$.

\medskip
We now prove the claim.  
Observe that 
\bqn
\Omega^{2m-1}(M\backslash P)^P\simeq\Lambda^{*(2m-1)}T_{x_0}(M\backslash P)^M\,,
\eqn
where $\Lambda^{*(2m-1)}T_{x_0}(M\backslash P)^M$ denotes the space of $M$-invariant $(2m-1)$-multilinear forms on 
the tangent space to $M\backslash P$ at the basepoint $x_0=[M]\in M\backslash P$.

Under the identification 
\bq\label{eq:id}
\ba
M\backslash P\quad&\longrightarrow\RR\times\RR^{2m-1}\\
Ma(t)n(x)&\longmapsto\,\,(t,x)
\ea
\eq
the point $x_0$ corresponds to $(0,0)$, so that 
\bqn
T_{x_0}(M\backslash P)\simeq T_{(0,0)}(\RR\times\RR^{2m-1})\simeq\RR\times\RR^{2m-1}
\eqn
and the $M$-action on $T_{x_0}(M\backslash P)$ corresponds to the action on $\RR\times\RR^{2m-1}$
via $\id\times\SO(2m-1)$
\bqn
\xymatrix{
(t,x)\ar[rr]^{\id\times m(U)}
&&(t,xU)\,.
}
\eqn
Let $\{e_0,e_1,\dots,e_{2m-1}\}$ be the canonical basis of $\RR\times\RR^{2m-1}$
(with $e_0$ spanning $\RR$) and let $\{e_0^\ast,e_1^\ast,\dots,e_{2m-1}^\ast\}$ be the dual basis.
Every $(2m-1)$-multilinear form $\alpha$ has a canonical decomposition
\bqn
\alpha=e_0^\ast\wedge\beta+\omega\,,
\eqn
where $\beta$ is a $(2m-2)$-multilinear form on $\RR^{2m-1}$ and 
$\omega$ is the pullback to $\RR\times\RR^{2m-1}$ of a $(2m-1)$-multilinear form on $\RR^{2m-1}$.
By uniqueness of the decomposition, $\alpha$ is $\id\times\SO(2m-1)$-invariant if and only if
$\beta$ and $\omega$ are $\SO(2m-1)$-invariant.  
Thus $\omega$ is a multiple of the determinant 
\bqn
\omega=\lambda e_1^\ast\wedge\dots\wedge\e_{2m-1}^\ast
\eqn
and if we show that there are no $\SO(2m-1)$-invariant $(2m-2)$-multilinear forms on $\RR^{2m-1}$
we will have shown that $\Omega^{2m-1}(M\backslash P)^P$ is one-dimensional.

The fact that $\Lambda^{\ast{2m-1}}(\RR^{2m-1})$ is one-dimensional and the pairing
\bqn
\ba
\Lambda^{\ast{2m-2}}(\RR^{2m-1})\times(\RR^{2m-1})^\ast&\to\Lambda^{\ast{2m-1}}(\RR^{2m-1})\\
(\beta,\lambda)\qquad\qquad&\longmapsto\quad\beta\wedge\lambda
\ea
\eqn
show that there is an isomorphism
\bqn
\ba
\Lambda^{\ast{2m-2}}(\RR^{2m-1})\qquad&\longrightarrow (\RR^{2m-1})^\ast\\
e_1^\ast\wedge\dots\wedge \widehat{e_j^\ast}\wedge\dots \wedge e_{2m-1}^\ast&\mapsto \qquad e_j^\ast
\ea
\eqn
that is $\SO(2m-1)$-equivariant.  Since the $\SO(2m-1)$-action on $(\RR^{2m-1})^\ast$ is irreducible,
there are no $\SO(2m-1)$-invariant $(2m-2)$-multilinear forms on $\RR^{2m-1}$,
thus showing that $\Omega^{2m-1}(M\backslash P)^P$ is one-dimensional.

We show now that the exterior derivative on $\Omega^{2m-1}(M\backslash P)^P$ does not vanish.
With the identification \eqref{eq:id}, the right translation $R_{(t,x)}$ by an element $m(U)a(t)n(x)$ is given by
\bqn
R_{(t,x)}(s,y)=(s+t,e^tyU+x)\,.
\eqn
In fact
\bqn
\ba
  R_{(t,x)}(s,y)
&=(M a(s)n(y) )m(U)a(t)n(x)\\
&=\underbrace{M m(U)}_{=M}\underbrace{m(U)^{-1}a(s)}_{=a(s)m(U)^{-1}}n(y)m(U)a(t)n(x)\\
&=Ma(s)\underbrace{m(U)^{-1}n(y)m(U)}_{=n(yU)}a(t)n(x)\\
&=Ma(s)n(yU)a(t)n(x)\\
&=Ma(s)a(t)\underbrace{a(t)^{-1}n(yU)a(t)}_{=n(e^tyU)}n(x)\\
&=Ma(s+t)n(e^tyU+x)\,.
\ea
\eqn
In particular
\bqn
R_{(t,x)}(0,0)=Ma(t)n(x)\simeq(t,x),
\eqn
so that $\omega\in\Omega^{2m-1}(\RR\times\RR^{2m-1})$ can be extended to a $P$-invariant differential $(2m-1)$-form
\bqn
\omega_{(0,0)}((s_1,y_1),\dots,(s_{2m-1},y_{2m-1}))=((R_{(t,x)})^\ast\omega)_{(0,0)}((s_1,y_1),\dots,(s_{2m-1},y_{2m-1}))\,
\eqn
for $(s_j,t_j)\in\RR\times\RR^{2m-1}\simeq T_{(0,0)}(\RR\times\RR^{2m-1})$, $j=1,\dots,2m-1$.

Since 
\bqn
\ba
{}&\omega_{(0,0)}((s_1,y_1),\dots,(s_{2m-1},y_{2m-1}))\\
&=\omega_{(t,x)}(D_{(0,0)}R_{(t,x)}((s_1,y_1)),\dots,D_{(0,0)}R_{(t,x)}((s_{2m-1},y_{2m-1})))\\
&=\omega_{(t,x)}((s_1,e^ty_1),\dots,(s_{2m-1},e^ty_{2m-1}))\\
&=e^{(2m-1)t}\omega_{(t,x)}((s_1,y_1),\dots,(s_{2m-1},y_{2m-1}))\,,
\ea
\eqn
then 
\bqn
\omega_{(t,x)}=e^{-(2m-1)t}\omega_{(0,0)}\,,
\eqn
so that 
\bqn
d\omega=-(2m-1)e^{-(2m-1)t}dt\wedge \omega
\eqn
is not vanishing.  This shows the claim and completes the proof of the lemma.
\end{proof}

\begin{lemma}\label{lem:vanish}  If $m\geq2$ then the restriction
\bqn
\vemb|_P\in\hhb^{2m}(P,\ZZ)
\eqn
vanishes.
\end{lemma}

\begin{proof} Considering the long exact sequence in bounded and ordinary cohomology
associated to \eqref{eq:0.8-},
and taking into account Lemma~\ref{lem:cohom of parabolic} and the vanishing 
of bounded real cohomology of $P$ (\ref{eq:06}), we obtain the diagram
\bqn
\xymatrix{
  0\ar[r]
&\h_\mathrm{B}^{2m-1}(P,\RR/\ZZ)\ar[r]^-{\delta^{\mathrm{b}}}\ar@{=}[d]
&\hhcb^{2m}(P,\ZZ)\ar[r]\ar[d]^{c_\ZZ}
&0\\
  0\ar[r]
&\h_\mathrm{B}^{2m-1}(P,\RR/\ZZ)\ar[r]_-\delta
&\hhc^{2m}(P,\ZZ)\ar[r]
&0\,,
}
\eqn
where we used Lemma~\ref{lem:cohom of parabolic} in the bottom row and
the vanishing of bounded cohomology with real coefficients in the top row.
Thus $\hhcb^{2m}(P,\ZZ)\cong\hhc^{2m}(P,\ZZ)$.
If we show that $\vem|_P=0$, this will imply that $\vemb|_P=0$.

To see that $\vem|_P=0$, by naturality of Wigner's isomorphism \eqref{eq:0.3} and the fact that $M$ is maximal compact in $P$,
the restriction to $M$ induces an isomorphism
$\hhc^{2m}(P,\ZZ)\cong\hhc^{2m}(M,\ZZ)$.
Since however the Euler class $\vem$ of $\SO(2m)$ restricted to 
$\left\{\begin{pmatrix}1&0\\0&U\end{pmatrix}:\,U\in\SO(2m-1)\right\}$
vanishes, we conclude that  $\vem|_P=0$.
\end{proof}

We deduce then using Lemma~\ref{lem:vanish} that Theorem~\ref{thm:vanishing} holds 
if the image of $\rho$ is contained in $P$.  We must therefore turn to the case
in which $\rho(\ZZ^{2m-1})$ lies in $T_0$.  Let then
\bqn
\pi_i\colon T\longrightarrow\O(2)
\eqn
be the projection on the $i$-th factor of $T$ and let 
\bqn
\vei:=\pi_i^*(\ve_2)\quad\text{ and }\quad\veib:=\pi_i^*(\ve_2^{\mathrm{b}})\,,
\eqn
where $\ve_2$ and $\ve_2^{\mathrm{b}}$ are respectively  the Euler class and the bounded
Euler class of $\O(2)$.  Observe that if $L$ is compact, the long exact sequence \eqref{eq:0.7} gives
\bq\label{eq:lulodiagram}
\xymatrix{
 0\ar[r]
& \h_\mathrm{B}^\bullet(L,\Rep/\Zep)\ar[r]^\cong\ar@{=}[d]
&\hhcb^\bullet(L,\Zep)\ar[d]\ar[r]
&0\\
 0\ar[r]
&\h_\mathrm{B}^\bullet(L,\Rep/\Zep)\ar[r]^\cong
&\hhc^\bullet(L,\Zep)\ar[r]
&0\,.
}
\eq
Since the ordinary Euler class is a characteristic class and $T$ is a product, we have
\bq\label{eq:eucup}
\vem|_T=\ve_{(1)}\abxcup\dots\abxcup\ve_{(m)}\,,
\eq
and hence it follows from \eqref{eq:lulodiagram} that
\bq\label{eq:eumcup}
\vemb|_T=\ve_{(1)}^{\mathrm{b}}\abxcup\dots\abxcup\ve_{(m)}^{\mathrm{b}}\,.
\eq
Observe that the image of $\SO(2)^m$ in $T$ is its connected component $T^\circ$.

\begin{lemma}\label{lem:vanish2}  If $\rho\colon \ZZ^{2m-1}\to T_0$ does not take values in $T^\circ$, 
then $\rho^*(\vemb)=0$.
\end{lemma}

\begin{proof}  An Abelian subgroup of $\O(2)$ not contained in $\SO(2)$ is of the form $\{1,\sigma\}$, with $\sigma^2=e$.  
Thus if $\rho(\ZZ^{2m-1})\not\subset T^\circ$,
there is $\pi_i$ such that $\pi_i\rho(\ZZ^{2m-1})\subset\{1,\sigma\}$.  Since $\hb^2(\{1,\sigma\},\Zep)=0$, 
then $\ve_2^\mathrm{b}$ vanishes on $\{1,\sigma\}$, and hence $\rho^*(\vemb)=0$ by \eqref{eq:eumcup}.
\end{proof}

Thus we are reduced to analyze homomorphisms 
\bqn
\rho\colon \ZZ^{2m-1}\longrightarrow T^\circ\cong\SO(2)^m\,.
\eqn
As before, since $\SO(2)^m$ is compact and hence amenable, the connecting homomorphism 
\bqn
  \delta^\mathrm{b}\colon \h_\mathrm{B}^{2m-1}(\SO(2)^m,\RR/\ZZ)\longrightarrow
  \hhcb^{2m}(\SO(2)^m,\ZZ)
\eqn
is an isomorphism. 
We fix the orientation preserving 
identification 
\bqn
\SO(2)\longrightarrow\RR/\ZZ\,,
\eqn
and, for $m=1$, we define the homogeneous $1$-cocycle
\bq \label{eq:rot}
\ba
\rot\colon \SO(2)\times\SO(2)&\longrightarrow\SO(2)\cong\RR/\ZZ\\
(g,h)\qquad&\longmapsto \qquad g^{-1}h
\ea
\eq
and
\bqn
[\rot]\in\h_\mathrm{B}^1(\SO(2),\RR/\ZZ)
\eqn
the corresponding class.  Then
$\delta^\mathrm{b}([\rot])=\ve_2^\mathrm{b}$ and from this we deduce that if
\bqn
\vartheta:=\pi_1^*([\rot])\in\h_\mathrm{B}^1(\SO(2)^m,\RR/\ZZ)\,,
\eqn
then $\delta^\mathrm{b}$ maps the class 
$\vartheta\abxcup\ve_{(2)}\abxcup\dots\abxcup\ve_{(m)}\in\h_\mathrm{B}^{2m-1}(\SO(2)^m,\RR/\ZZ)$
to the class $\ve_{(1)}^\mathrm{b}\abxcup\dots\abxcup\ve_{(m)}^\mathrm{b}\in\hhcb^{2m}(\mathrm{SO}(2)^m,\ZZ)$.
As a result, it follows from the commutativity of the square
\bqn
\xymatrix{
  \h_\mathrm{B}^{2m-1}(\SO(2)^m,\RR/\ZZ)\ar[r]^-{\delta^\mathrm{b}}\ar[d]_{\rho^\ast}
&\hhcb^{2m}(\SO(2)^m,\ZZ)\ar[d]^{\rho^*}\\
  \h^{2m-1}(\ZZ^{2m-1},\RR/\ZZ)\ar[r]^-{\delta^\mathrm{b}}
&\hb^{2m}(\ZZ^{2m-1},\ZZ)
}
\eqn
that 
\bq\label{eq:2.4}
\rho^*(\vemb)=\delta^\mathrm{b}(\rho^*(\vartheta\abxcup\ve_{(2)}\abxcup\dots\abxcup\ve_{(m)}))\,.
\eq
In order to prove the vanishing of the left hand side, we are going to show that 
$\rho^*(\vartheta\abxcup\ve_{(2)}\abxcup\dots\abxcup\ve_{(m)})=0$.  To this end we use that the pairing
\bq\label{eq:pairing}
\ba
\h^{2m-1}(\ZZ^{2m-1},\RR/\ZZ)&\longrightarrow\qquad\,\,\RR/\ZZ\\
\beta\qquad\qquad&\longmapsto\,\,\langle\beta,[\ZZ^{2m-1}]\rangle
\ea
\eq
is an isomorphism, 
where $[\ZZ^{2m-1}]\in\h_{2m-1}(\ZZ^{2m-1},\ZZ)\cong\ZZ$ denotes the fundamental class.

We will need the following

\begin{lemma}\label{lem: representation Zn}  Let $n$ be even and $n\geq2$ and let $e_1,\dots,e_n$ be the canonical basis of $\ZZ^n$.
Then the group chain
\bqn
z= \sum_{\sigma\in \mathrm{Sym}(n)} \mathrm{sign}(\sigma)[0,e_{\sigma(1)},e_{\sigma(1)}+e_{\sigma(2)},\dots,e_{\sigma(1)}+\dots+e_{\sigma(n)}]
\eqn
is a representative of the fundamental class  $[\ZZ^n]\in \h_n(\ZZ^n,\ZZ)\cong \ZZ$.
\end{lemma}

\begin{proof} We first check that $\partial z=0$ so that $z$ is indeed a cycle. 
Let 
\bqn
s_\sigma:=[0,e_{\sigma(1)},e_{\sigma(1)}+e_{\sigma(2)},\dots,e_{\sigma(1)}+\dots+e_{\sigma(n)}]\,.
\eqn
Observe that, for $1\leq i\leq n-1$, the sum of each of the $i$-th faces of the simplices appearing in $z$ 
over all permutations of $\sigma\in \mathrm{Sym}(n)$ vanishes since
\bqn
\partial_is_\sigma=[0,e_{\sigma(1)},e_{\sigma(1)}+e_{\sigma(2)},\dots,\wideparen{e_{\sigma(1)}+\dots+e_{\sigma(i)}},\dots,e_{\sigma(1)}+\dots+e_{\sigma(n)}]
\eqn
is invariant under the odd transposition $\tau_i:=(\sigma(i),\sigma(i+1))$. 
Indeed if $1\leq i\leq n-1$,
\bqn
\ba
\sum_{\sigma\in \mathrm{Sym}(n)} \mathrm{sign}(\sigma)\partial_is_\sigma
&=\sum_{\sigma\in \mathrm{Sym}(n)}\mathrm{sign}(\sigma)\partial_is_{\sigma\tau_i}\\
&=\sum_{\sigma\in \mathrm{Sym}(n)}\mathrm{sign}(\sigma\tau_i^{-1})\partial_is_{\sigma}\\
&=-\sum_{\sigma\in \mathrm{Sym}(n)}\mathrm{sign}(\sigma)\partial_is_\sigma
\ea
\eqn
and hence implies that $\sum_{\sigma\in \mathrm{Sym}(n)}\mathrm{sign}(\sigma)\partial_is_\sigma=0$.
Thus
\bq\label{eq:5.75}
\partial z=\sum_{\sigma\in \mathrm{Sym}(n)}\mathrm{sign}(\sigma)\partial_0s_\sigma+\sum_{\sigma\in \mathrm{Sym}(n)}\mathrm{sign}(\sigma)\partial_ns_\sigma
\eq
and it remains to see that the $0$-th and $n$-th face cancel each other. 
By $\ZZ^n$-invariance,
\bq\label{eqn:0andnFaces}
\ba
\partial_0s_\sigma
&=\small{[e_{\sigma(1)},e_{\sigma(1)}+e_{\sigma(2)},\dots,e_{\sigma(1)}+\dots+e_{\sigma(n)}]}\\
&=\small{[0,e_{\sigma(2)},\dots,e_{\sigma(2)}+\dots+e_{\sigma(n)}]}\\
&=\small{[0,e_{\sigma\tau(1)},\dots,e_{\sigma\tau(1)}+\dots+e_{\sigma\tau(n-1)}]}\\
&=\partial_n s_{\sigma\tau}\,,
\ea
\eq
where $\tau=(1,2,\dots,n)$ is the cyclic permutation of signature $\mathrm{sign}(\tau)=(-1)^{n-1}$. 
From \eqref{eq:5.75} and \eqref{eqn:0andnFaces} follows that
\bqn
\ba
\partial z
&=\sum_{\sigma\in \mathrm{Sym}(n)}\mathrm{sign}(\sigma)\partial_ns_{\sigma\tau}+\sum_{\sigma\in \mathrm{Sym}(n)}\mathrm{sign}(\sigma)\partial_ns_{\sigma}\\
&=\sum_{\sigma\in \mathrm{Sym}(n)}(-1)^{n-1}\mathrm{sign}(\sigma)\partial_ns_{\sigma}+\sum_{\sigma\in \mathrm{Sym}(n)}\mathrm{sign}(\sigma)\partial_ns_{\sigma}
=0
\ea
\eqn
since $n$ is even.

Let $\omega_{\RR^n}\in\h^n(\ZZ^n,\RR)$ be the Euclidean volume class. 
Note that the volume class evaluates to $1$ on the fundamental class, 
as the $n$-torus generated by the canonical basis has volume $1$. 
A cocycle representing $\omega_{\RR^n}$ is given by $V_n\colon (\ZZ^n)^{n+1}\rightarrow \RR$ 
sending $v_0,v_1,\dots,v_n$ to the signed volume of the affine simplex 
with vertices in the lattice $\ZZ^n\subset \RR^n$, that is 
\bqn
V_n(v_0,v_1,\dots,v_n)=\frac{1}{n!} \mathrm{det}(v_1-v_0,\dots,v_n-v_0)\,.
\eqn
In order to show that 
\bqn
\langle \omega_{\RR^n}, [\ZZ^n]\rangle=V_n(z)=1
\eqn 
we need to evaluate $V_n$ on each summand of $z$. 
By definition,
\bqn
\ba
&V_n(0,e_{\sigma(1)},e_{\sigma(1)}+e_{\sigma(2)},\dots,e_{\sigma(1)}+\dots+e_{\sigma(n)})\\
=&\frac{1}{n!}\det(e_{\sigma(1)},e_{\sigma(1)}+e_{\sigma(2)},\dots,e_{\sigma(1)}+\dots+e_{\sigma(n)})\\
=&\mathrm{sign}(\sigma)\frac{1}{n!}\det(e_1,e_1+e_2\dots,e_1+\dots+e_n)\\
=&\mathrm{sign}(\sigma)\frac{1}{n!}\,,
\ea
\eqn
where we have used for the second equality the fact that the determinant is alternating with respect to line permutations. 
Summing up, we obtain
\bqn
\langle \omega_{\RR^n} , [\ZZ^n] \rangle = \sum_{\sigma\in \mathrm{Sym}(n)} \frac{1}{n!}=1\,,
\eqn
thus proving the lemma. 
\end{proof}

\begin{lemma}\label{lem:rot=0}  Let $m\geq2$.  Then 
\bqn
\langle\rho^*(\vartheta\abxcup\ve_{(2)}\abxcup\dots\abxcup\ve_{(m)}),[\ZZ^{2m-1}]\rangle=0\,.
\eqn
\end{lemma}

\begin{proof}  
We use as a representative of $\ve_2\in\h^2(\mathrm{SO}(2),\ZZ)$ the multiple $-\frac12$ of the orientation cocycle
(see Remark~\ref{rem:3.1}).
This cocycle takes values in $\frac{1}{2}\mathbb{Z}$ but represents an integral class 
and in particular evaluates to an integer on a fundamental class.
Hence a representative for $\kappa:=\vartheta\abxcup \ve_{(2)}\abxcup \dots \abxcup \ve_{(m)}\in\h_\mathrm{B}^{2m-1}(\SO(2)^m,\RR/\ZZ)$ 
is given by the cocycle mapping $g_0,\dots,g_{2m-1}\in\SO(2)^m $ to the product
\bqn
\frac{(-1)^{m-1}}{2^{m-1}}\cdot  \mathrm{Rot}_1(g_0,g_1)\cdot \mathrm{Or}_2(g_1,g_2,g_3)\cdot \dots \cdot \mathrm{Or}_m(g_{2m-3},g_{2m-2},g_{2m-1})\,,
\eqn
where $\mathrm{Rot}_1$ denotes the pullback to $\SO(2)^m$ via the first projection $\pi_1$ of the homogeneous 
1-cocycle $\rot$ defined in (\ref{eq:rot}), while $\mathrm{Or}_j$ denotes the pullback to $\SO(2)^m$ via $\pi_j$ of the orientation cocycle
\bqn
\mathrm{Or}\colon \SO(2)^3\longrightarrow\{-1,0,1\}\,.
\eqn
We now evaluate the pullback $\rho^*(\kappa)$ on the cycle $z$ of Lemma~\ref{lem: representation Zn} and, 
writing $f_i=\rho(e_i)\in  \SO(2)^m$, obtain
\bq \label{equ: rho on torus}
\ba
&\langle\rho^*(\kappa),[\ZZ^{2m-1}]\rangle\\
=&\frac{(-1)^{m-1}}{2^{m-1}} \sum_{\sigma\in \mathrm{Sym}(n)} \mathrm{sign}(\sigma) \mathrm{Rot}_1(0,f_{\sigma(1)}) \cdot\mathrm{Or}_2(f_{\sigma(1)}, f_{\sigma(1)}f_{\sigma(2)},f_{\sigma(1)} f_{\sigma(2)}f_{\sigma(3)})\\
&\hphantom{XXXXxXXXX}\cdot (\mathrm{Or}_3\abxcup \dots \abxcup \mathrm{Or}_m)(f_{\sigma(1)} f_{\sigma(2)}f_{\sigma(3)},\dots,f_{\sigma(1)}\cdot \dots\cdot f_{\sigma(n)}).
\ea
\eq

To prove that this expression vanishes first observe that 
\bqn
\mathrm{Or}(f_1,f_2,f_3)=-\mathrm{Or}(f_1^{-1},f_2^{-1},f_3^{-1})
\eqn
holds for any $f_1,f_2,f_3$ in $\SO(2)$ (but not in $\SL(2,\mathbb{R})$). 
Specializing to $f_1=g^{-1}$, $f_2=\mathrm{Id}$, $f_3=h$ and using that $\mathrm{Or}$ is alternating gives
\bqn
\mathrm{Or}(g^{-1},\mathrm{Id},h)=\mathrm{Or}(h^{-1},\mathrm{Id},g)
\eqn
for any $g,h$ in $\SO(2)$. Finally, using the invariance of $\mathrm{Or}$, 
we can multiply the variables on the left hand side of the latter equality by $fg$, and the variables on the right hand side by $fh$ to obtain
\bq \label{equ:Or(f,fg,fgh)}
\mathrm{Or}(f,fg,fgh)=\mathrm{Or}(f,fh,fhg)
\eq
for any $f,g,h$ in $\SO(2)$. 

Thus from \eqref{equ: rho on torus} and \eqref{equ:Or(f,fg,fgh)} and with $\tau_2=(2 3)$, we obtain the asserted vanishing as
\bqn
\ba
&\langle\rho^*(\kappa),[\ZZ^{2m-1}]\rangle\\
=&\frac{(-1)^{m-1}}{2^{m-1}} \sum_{\sigma\in \mathrm{Sym}(n)} \mathrm{sign}(\sigma) \mathrm{Rot}_1(0,f_{\sigma(1)}) \cdot\mathrm{Or}_2(f_{\sigma(1)}, f_{\sigma(1)}f_{\sigma(3)},f_{\sigma(1)} f_{\sigma(3)}f_{\sigma(2)})\\
&\hphantom{XXXXxXXXX}\cdot (\mathrm{Or}_3\abxcup \dots \abxcup \mathrm{Or}_m)(f_{\sigma(1)} f_{\sigma(2)}f_{\sigma(3)},\dots,f_{\sigma(1)}\cdot \dots\cdot f_{\sigma(n)})\\
=&\mathrm{sign}(\tau_2)\frac{(-1)^{m-1}}{2^{m-1}} \sum_{\sigma\in \mathrm{Sym}(n)} \mathrm{sign}(\sigma) \mathrm{Rot}_1(0,f_{\sigma(1)}) \cdot\mathrm{Or}_2(f_{\sigma(1)}, f_{\sigma(1)}f_{\sigma(2)},f_{\sigma(1)} f_{\sigma(2)}f_{\sigma(3)})\\
&\hphantom{XXXXxXXXX}\cdot (\mathrm{Or}_3\abxcup \dots \abxcup \mathrm{Or}_m)(f_{\sigma(1)} f_{\sigma(2)}f_{\sigma(3)},\dots,f_{\sigma(1)}\cdot \dots\cdot f_{\sigma(n)})\\
=&-\langle\rho^*(\kappa),[\ZZ^{2m-1}]\rangle\,.
\ea
\eqn

\end{proof}

\begin{proof}[Proof of Theorem~\ref{thm:vanishing}]  Let $\rho\colon \ZZ^{2m-1}\to\SO(2m,1)^\circ$ be a homomorphism.  
By Lemma~\ref{lem:tricot} either $\rho(\ZZ^{2m-1})<P$ up to conjugation and then $\rho^*(\vemb)=0$
by Lemma~\ref{lem:vanish} or $\rho^*(\ZZ^{2m-1})<T_0$ up to conjugacy. 
Then either $\rho(\ZZ^{2m-1})\not\subset T^\circ$ and the vanishing follows from Lemma~\ref{lem:vanish2} or
$\rho(\ZZ^{2m-1})<T^\circ$, in which case Lemma~\ref{lem:rot=0} and \eqref{eq:2.4} imply that 
$\rho^*(\vartheta\abxcup\ve_{(2)}\abxcup\dots\abxcup\ve_{(m)})=0$ and hence $\rho^*(\vemb)=0$ by \eqref{eq:pairing}.
\end{proof}


\begin{proof}[Proof of Theorem \ref{thm: main thm}] 
%
Let $N$ be a compact core of $M$ and let  $C_1,\dots,C_h$ the connected components of the boundary of $\partial N$.
It follows from Theorem~\ref{thm: vol modulo boundary} that 
\bqn
(-1)^m\frac{2}{\Vol(S^{2m})}\Vol(\rho)\equiv-\sum_{i=1}^h \langle (\delta^\mathrm{b})^{-1}(\rho^*(\vemb)|_{C_i},[C_i]\rangle\mod \ZZ\,,
\eqn
with the usual abuse of notation that $\rho^*(\vemb)|_{C_i}$ refers to the element in $\hb^{2m}(C_i,\ZZ)$ 
corresponding to $(\rho|_{\pi_1(C_i)})^*(\vemb)\in\hb^{2m}(\pi_1(C_i),\ZZ)$.

If now all the $C_i$'s are tori, the above congruence relation and Theorem~\ref{thm:vanishing} imply
that $(-1)^m\frac{2}{\Vol(S^{2m})}\Vol(\rho)\in\ZZ$.

In the general case, let $p_i\colon C'_i\to C_i$ be a covering of degree $B_{2m-1}$ that is a torus.
Then
\bqn
\ba
B_{2m-1}\langle(\delta^\mathrm{b})^{-1}(\rho^*(\vemb)|_{C_i}),[C_i]\rangle
=&\langle(\delta^\mathrm{b})^{-1}(\rho^*(\vemb)|_{C_i}),p_{i*}([C'_i])\rangle\\
=&\langle(\delta^\mathrm{b})^{-1}p_i^*(\rho^*(\vemb)|_{C_i}),[C'_i]\rangle\,.
\ea
\eqn
Now observe that $p_i^*(\rho^*(\vemb)|_{C_i})\in\hb^{2m}(C'_i,\ZZ)$
corresponds to the class
\bqn
(\rho\circ p_{i*})^*(\vemb)=(\rho|_{\pi_1(C'_i))}^*(\vemb)\in\hb^{2m}(\pi_1(C'_i),\ZZ)\,,
\eqn
which vanishes by Theorem~\ref{thm:vanishing}.
\end{proof}

\section{Examples of nontrivial and non-maximal representations}\label{sec:examples}
In this section we give examples of volumes of representations.  More precisely:
\be
\item[--] In \S~\ref{subsec:5.1} we set ourselves in dimension $3$.  Here we show in particular 
that the volume of a Dehn filling of a finite volume hyperbolic manifold
coincides with the volume of the filling representation. In fact Proposition~\ref{prop: vol(rho)=vol(rho0)}
deals with a more general case.
\item[--] In \S~\ref{subsec:5.2}, by glueing appropriately copies of a hyperbolic manifold of arbitrary dimensions 
with totally geodesic boundary, we construct manifolds $M_k$ and representations of $\pi_1(M_k)$ 
whose volume is a rational multiple of $\Vol(M_k)$.
\ee

\subsection{Dimension $3$: representations given by Dehn filling}\label{subsec:5.1}

Let $M$ be a complete finite volume hyperbolic $3$-manifold, which, for simplicity, 
we assume has only one cusp. 
If $N$ is a compact core of $M$, its boundary $\partial N$ is Euclidean with the induced metric
and hence there is an isometry $\varphi\colon \partial N\to\TT^2$ to a two-dimensional torus 
for an appropriate flat metric on $\TT^2$.  We obtain then a decomposition of $M$ as a connected sum
\bqn
M=N\#(\TT^2\times\RR_{\geq0})\,,
\eqn
where the identification is via $\varphi$.
We are now going to fill in a solid two-torus to obtain a compact manifold.
To this end, let $\tau\subset\partial N$ be a simple closed geodesic and let us choose a diffeomorphism
$\varphi_\tau\colon \partial N\to S^1\times S^1$, in such a way that $\varphi_\tau(\tau)=S^1\times\{*\}$.
Then $M_\tau$ is the connected sum
\bqn
M_\tau\colon =N\#(\DD^2\times S^1)\,,
\eqn
identified via $\varphi_\tau$.

%

Denote by $j_\tau\colon N\hookrightarrow M_\tau $ the canonical inclusion and 
by $p\colon M\rightarrow N$ the canonical projection given by the cusp retraction 
$\TT^2\times \RR_{>0} \rightarrow \TT^2 $. 
The composition 
\bqn
f_\tau= j_\tau \circ p\colon  M \longrightarrow M_\tau
\eqn
induces a map 
\bqn
(f_\tau)_*\colon \Gamma \longrightarrow \Gamma_\tau
\eqn
between the fundamental groups $\Gamma=\pi_1(M)$ and $\Gamma_\tau=\pi_1(M_\tau)$.

\begin{prop} \label{prop: vol(rho)=vol(rho0)} Let $M_\tau$ be the compact $3$-manifold obtained
by Dehn filling from the hyperbolic $3$-manifold $M$ with one cusp.
Let $\rho\colon \Gamma_\tau\rightarrow \mathrm{SO}(3,1)$ 
be any representation of $\Gamma_\tau$ and let  $\rho_\tau\colon =\rho\circ f_\tau\colon \Gamma\to\mathrm{SO}(3,1)$.
Then
\bqn
\Vol(\rho_\tau)=\Vol(\rho)\,.
\eqn
\end{prop}

By Gromov--Thurston $(2\pi)$-Theorem \cite{GromovThurston}, 
for all geodesic curves $\tau$
for which the induced length is greater than $2\pi$ in the induced Euclidean metric on $\partial N$, 
the compact manifold $M_\tau$ admits a hyperbolic structure.
Proposition~\ref{prop:intro} is then an immediate consequence of Proposition~\ref{prop: vol(rho)=vol(rho0)}.


To prove the proposition, recall that by definition, the volume of the representation $\rho_\tau$ is equal to
\bqn
\Vol(\rho_\tau)=\langle c\circ \Psi^{-1} \circ f^* \circ \rho^*(\omega_3^{\mathrm{b}}), [N,\partial N]\rangle\,,
\eqn
where all maps involved can be read in the diagram below. 
We will start by defining a map $F\colon  \h^3(M_\tau)\rightarrow  \h^3(N,\partial N)$ 
that will turn the diagram below into a commutative diagram (Lemma \ref{lem: F commute}) 
and which will induce a canonical isomorphism (Lemma \ref{lem: F isom}). 
\begin{equation}\label{equ: big diag degree 3}
\xymatrix{ 
   \hcb^3(\mathrm{SO}(3,1)) \ar[d]^{c} \ar@/^2pc/[rr]^{\rho_\tau^*}  \ar[r]_-{\rho^*}\ar[d]^{c} 
&{\color{red}{\hb^3(\Gamma_\tau)}}\ar@[red][d]^{\color{red}{c}} \ar@[red][r]_{\color{red}{f_\tau^*}}
&{\color{red}{\hb^3(\Gamma)}}
&{\color{red}{\hb^3(N,\partial N)}}\ar@[red][d]^{\color{red}{c}}  \ar@[red][l]^-{\color{red}{\Psi}}_-{\color{red}{\cong}} \\
\hc^3(\mathrm{SO}(3,1)) \ar[r]_-{\rho^*}
&{\color{red}{\h^3(\Gamma_\tau)}}\ar@[red][r]^{\color{red}{\cong}}_{\color{red}{g}}
&{\color{red}{\h^3(M_\tau)}}\ar@{.>}@[red][r]_-{\color{red}{F\,\,}}
&{\color{red}{\h^3(N,\partial N)}}.}
\end{equation}

The inclusions 
\begin{equation*}
\xymatrix{
M_\tau \ar@{^{(}->}[dr]_i& & (N,\partial N)\ar@{_{(}->}[dl]^{(j_\tau,\varphi_\tau)} \\
&(M_\tau,\mathbb{D}^2\times S^1)&
}
 \end{equation*}
induce the following homology and cohomology maps
\begin{equation}\label{equ: ho can iso}
\xymatrix{
 \h_\bullet(M_\tau,\ZZ)\ar[dr]_-{i_*}
 & &\h _\bullet((N,\partial N),\ZZ)\ar[dl]^{(j_\tau,\varphi_\tau)_*} \\
 &\h_\bullet((M_\tau,\mathbb{D}^2\times S^1),\ZZ)
}
\end{equation}
 and
 \begin{equation}\label{equ: coho can iso}
 \xymatrix{
 \h^\bullet(M_\tau,\ZZ)
 & &\h ^\bullet((N,\partial N),\ZZ) \\
 &\h^\bullet((M_\tau,\mathbb{D}^2\times S^1),\ZZ)\ar[lu]^{i_*}\ar[ur]_{(j_\tau,\varphi_\tau)_*}
}
\end{equation}

 
 \begin{lemma} \label{lem: F isom} In degree $3$ the maps in \eqref{equ: ho can iso} and \eqref{equ: coho can iso} 
 are canonical isomorphisms and the composition 
 \bqn
 F=(j_\tau,\varphi_\tau)^*\circ (i^*)^{-1}\colon  \h^3(M_\tau,\ZZ) \longrightarrow \h^3(N,\partial N ,\ZZ)
 \eqn
 maps the dual $\beta_{M_\tau}$ of the fundamental class of $M_\tau$ 
 to the dual $\beta_{[N,\partial N]}$ of the fundamental class of $(N,\partial N)$.
 \end{lemma}
 
 \begin{proof} It is enough to show the statement in homology 
 where we show that fundamental classes are mapped to each other 
 by showing the existence of a compatible triangulation of the three manifolds. 
 Start with a triangulation of the boundary torus $S^1\times S^1 =\partial N$, 
 extend it on the one hand to the filled torus $\mathbb{D}^2\times S^1$ and on the other hand to $N$,
 \cite[Theorem 10.6]{Munkres}. 
 This produces compatible triangulations representing $[M_\tau]$,  
 $[M_\tau,\mathbb{D}^2\times S^1]$ and $[N,\partial N]$.
 \end{proof}

\begin{lemma}\label{lem: F commute} The diagram (\ref{equ: big diag degree 3}) commutes.
\end{lemma}

\begin{proof} We only need to show that the right rectangle commutes. For this, we will decompose the diagram in subdiagrams as follows:
\bqn
\xymatrix{ 
  {\color{red}\hb^3(\Gamma)}\ar[r]^\cong
&\hb^3(M)
&\hb^3(N)\ar[l]_\cong^{p^*}
& \\ 
{\color{red}\hb^3(\Gamma_\tau)}\ar@[red][u]^{\color{red}{f^*}}\ar[r]^\cong \ar@[red][dd]_{\color{red}{c}}
&\hb^3(M_\tau) \ar[u]^{f^*}\ar[ru]_-{j_\tau^*} \ar[dd]_{c}
&  
&{\color{red} \hb^3(N,\partial N)} \ar@[red]@/_5.3pc/[lllu]^{\color{red}{\Psi}} \ar[lu]_-{i_{| N}^*}\ \ar@[red][dd]^{\color{red}{c}}\\
& 
&\hb^3(M_\tau,\mathbb{D}^2\times S^1)  \ar[dd]^<<<<<<{c}\ar[lu]_-{i^*}\ar[ur]^-{(j_\tau,\varphi_\tau)^*}
& \\
{\color{red}\h^3(\Gamma_\tau)}\ar@[red][r]^{\color{red}\cong}_{\color{red}g}
&{\color{red}{ \h^3(M_\tau)} }\ar@[red]@{->}'[r][rr]^<<<<<<<<<<<<<{\color{red}{ F}}
&  
& {\color{red}\h^3(N,\partial N)}\\
&
&\h^3(M_\tau,\mathbb{D}^2\times S^1) \ar[lu]_-{i^*}\ar[ru]^-{(j_\tau,\varphi_\tau)^*}\,.
&
}
\eqn
Since by naturality, all subdiagrams commute, the lemma follows. 
\end{proof}

\begin{proof}[Proof of Proposition \ref{prop: vol(rho)=vol(rho0)}] 
Using the commutativity of the diagram (\ref{equ: big diag degree 3}) and the fact that 
$ c (\omega_3^{\mathrm b})= \omega_3$ and 
$g\circ\rho^*(\omega_3)=\Vol(\rho) \cdot \beta_{M_\tau}$ we compute
\begin{eqnarray*} 
c\circ \Psi^{-1}\circ f_\tau^*\circ \rho^*(\omega_3^{\mathrm b})
&=&  F\circ g\circ \rho^*\circ c (\omega_3^{\mathrm b})\\
&=& F\circ g\circ \rho^*(\omega_3)\\
&=& F(\Vol(\rho) \cdot \beta_{M_\tau}) \\
&=& \Vol(\rho) \cdot \beta_{[N,\partial N]}.
\end{eqnarray*}
It is immediate that
\bqn
\Vol(\rho_\tau)=\langle c\circ \Psi^{-1} \circ f^* \circ \rho^*(\omega_3)^{\mathrm{b}}, [N,\partial N] \rangle 
= \langle  \Vol(\rho) \cdot \beta_{[N,\partial N]}, [N,\partial N]\rangle =\Vol(\rho)\,,
\eqn
which finishes the proof of the proposition.
\end{proof}

\subsection{Representations giving rational multiples of the maximal representation}\label{subsec:5.2}

Let $M$ be a $n$-dimensional hyperbolic manifold with nonempy totally geodesic boundary 
(possibly with cusps) (see for example \cite{Millson}).
Suppose that the boundary of $M$ has at least two connected components, 
which can be achieved by taking appropriate coverings of a manifold with a connected boundary.
Decompose $\partial M=C_0 \sqcup C_1$. Let $M'$ be the double of $M$ along $C_1$. 
Observe that $M'$ has as boundary two copies of $C_0$ with opposite orientation. 
Glueing these two copies, we obtain a complete hyperbolic manifold $M_1$. 
We can repeat the procedure as follows: Take $k$ copies of $M'$, 
glue the two copies of $C_0$ two by two so as to obtain a connected closed hyperbolic manifold $M_k$. 
Observe that $\Vol(M_k)=2k\,\Vol(M)$. 

\vskip1cm
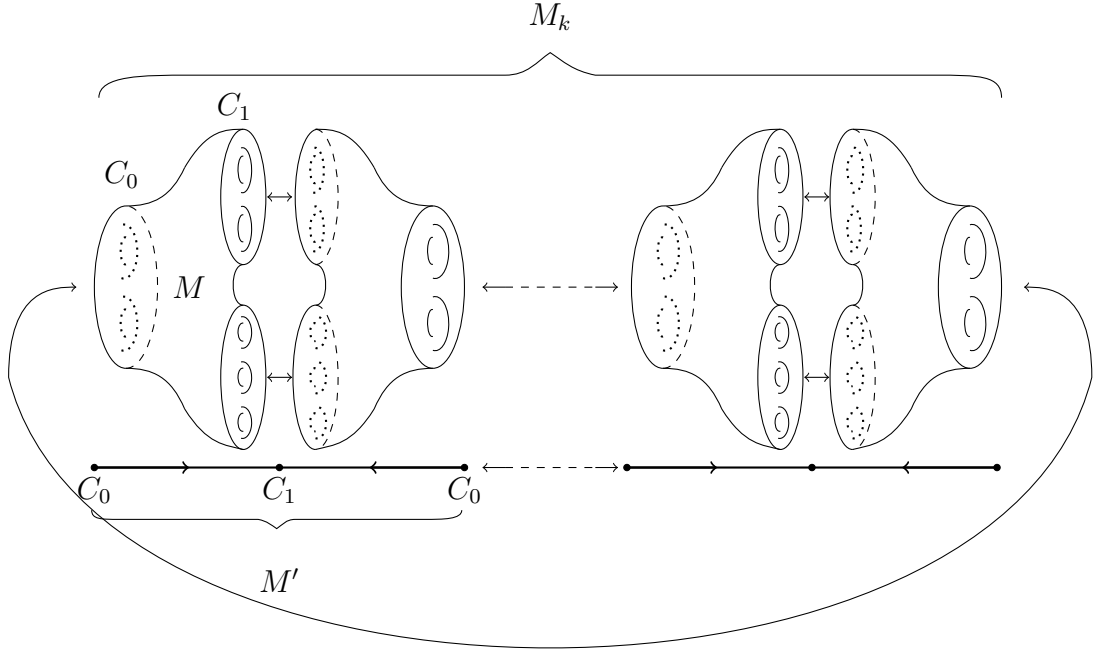
\begin{tikzpicture}
[scale=.6]
\draw (-3.4,1.8) arc (90:270: .7cm and 1.8cm);						\draw (3.4,1.8) arc (90:270: .7cm and 1.8cm);
 \draw [xshift=2cm](6.5,1.8) arc (90:270: .7cm and 1.8cm);						\draw[xshift=2cm] (13.3,1.8) arc (90:270: .7cm and 1.8cm);
\draw[dashed] (-3.4,-1.8) arc (-90:90: .7cm and 1.8cm);			\draw (3.4,-1.8) arc (-90:90: .7cm and 1.8cm);
\draw[dashed] [xshift=2cm](6.5,-1.8) arc (-90:90: .7cm and 1.8cm);			\draw [xshift=2cm](13.3,-1.8) arc (-90:90: .7cm and 1.8cm);
\draw[dotted, thick] (-3.5,.2) arc (-100:100: .3cm and .6cm); 						\draw (3.4,.2) arc (-100:100: .3cm and .6cm); 
\draw[dotted, thick] [xshift=2cm](6.5,.2) arc (-100:100: .3cm and .6cm); 						\draw [xshift=2cm](13.3,.2) arc (-100:100: .3cm and .6cm); 
\draw[dotted, thick] (-3.4,1.1) arc (80:280: .1cm and .3cm);						\draw (3.4,1.1) arc (80:280: .1cm and .3cm);
\draw[dotted, thick] [xshift=2cm](6.5,1.1) arc (80:280: .1cm and .3cm);						\draw [xshift=2cm](13.3,1.1) arc (80:280: .1cm and .3cm);
\draw[dotted, thick] (-3.5,-1.4) arc (-100:100: .3cm and .6cm); 					\draw (3.4,-1.4) arc (-100:100: .3cm and .6cm); 
\draw[dotted, thick] [xshift=2cm](6.5,-1.4) arc (-100:100: .3cm and .6cm); 					\draw [xshift=2cm](13.3,-1.4) arc (-100:100: .3cm and .6cm); 
\draw[dotted, thick] (-3.4,-.5) arc (80:280: .1cm and .3cm);							\draw (3.4,-.5) arc (80:280: .1cm and .3cm);
\draw[dotted, thick] [xshift=2cm](6.5,-.5) arc (80:280: .1cm and .3cm);								\draw [xshift=2cm](13.3,-.5) arc (80:280: .1cm and .3cm);
\draw (-.8,2) ellipse (.5cm and 1.5cm);									\draw[dashed] (.8,3.5) arc (90:-90: .5cm and 1.5cm);
																							\draw (.85,3.5) arc (90:270: .5cm and 1.5cm);
\draw [xshift=2cm](9.1,2) ellipse (.5cm and 1.5cm);									\draw[dashed] [xshift=2cm](10.7,3.5) arc (90:-90: .5cm and 1.5cm);
																							\draw [xshift=2cm](10.7,3.5) arc (90:270: .5cm and 1.5cm);
\draw (-.85,2.1) arc (-100:100:.2cm and .5cm);						\draw[dotted, thick] (.75,2.1) arc (-100:100:.2cm and .5cm);
\draw [xshift=2cm](9.05,2.1) arc (-100:100:.2cm and .5cm);						\draw[dotted, thick] [xshift=2cm](10.7,2.1) arc (-100:100:.2cm and .5cm);
\draw (-.8,2.9) arc (80:280:.1cm and .3cm);							\draw[dotted, thick] (.8,2.9) arc (80:280:.1cm and .3cm);
\draw [xshift=2cm](9.1,2.9) arc (80:280:.1cm and .3cm);							\draw[dotted, thick] [xshift=2cm](10.7,2.9) arc (80:280:.1cm and .3cm);
\draw (-.85,.8) arc (-100:100: .2cm and .5cm);						\draw[dotted, thick] (.75,.8) arc (-100:100: .2cm and .5cm);
\draw [xshift=2cm](9.05,.8) arc (-100:100: .2cm and .5cm);						\draw[dotted, thick][xshift=2cm] (10.65,.8) arc (-100:100: .2cm and .5cm);
\draw (-.8,1.6) arc (80:280:.1cm and .3cm);							\draw[dotted, thick] (.8,1.6) arc (80:280:.1cm and .3cm);
\draw [xshift=2cm](9.1,1.6) arc (80:280:.1cm and .3cm);							\draw[dotted, thick] [xshift=2cm](10.7,1.6) arc (80:280:.1cm and .3cm);
\draw (-.8,-2) ellipse (.5cm and 1.6cm);									\draw[dashed] (.8,-.4) arc (90:-90:.5cm and 1.6cm);
																							\draw (.8,-.4) arc (90:270:.5cm and 1.6cm);
\draw [xshift=2cm](9.1,-2) ellipse (.5cm and 1.6cm);								\draw[dashed] [xshift=2cm](10.7,-.4) arc (90:-90:.5cm and 1.6cm);
																							\draw [xshift=2cm](10.7,-.4) arc (90:270:.5cm and 1.6cm);
\draw (-.85,-1.35) arc (-100:100:.2cm and .35cm);					\draw[dotted, thick] (.8,-1.35) arc (-100:100:.2cm and .35cm);
\draw [xshift=2cm](9.05,-1.35) arc (-100:100:.2cm and .35cm);				\draw[dotted, thick] [xshift=2cm](10.7,-1.35) arc (-100:100:.2cm and .35cm);
\draw (-.8,-.8) arc (80:280:.1cm and .2cm);							\draw[dotted, thick] (.8,-.8) arc (80:280:.1cm and .2cm);
\draw [xshift=2cm](9.1,-.8) arc (80:280:.1cm and .2cm);							\draw[dotted, thick] [xshift=2cm](10.7,-.8) arc (80:280:.1cm and .2cm);
\draw (-.85,-2.35) arc (-100:100:.2cm and .35cm);					\draw[dotted, thick] (.8,-2.35) arc (-100:100:.2cm and .35cm);
\draw [xshift=2cm](9.05,-2.35) arc (-100:100:.2cm and .35cm);				\draw[dotted, thick] [xshift=2cm](10.7,-2.35) arc (-100:100:.2cm and .35cm);
\draw (-.8,-1.8) arc (80:280:.1cm and .2cm);							\draw[dotted, thick] (.8,-1.8) arc (80:280:.1cm and .2cm);
\draw [xshift=2cm](9.1,-1.8) arc (80:280:.1cm and .2cm);							\draw[dotted, thick] [xshift=2cm](10.7,-1.8) arc (80:280:.1cm and .2cm);
\draw (-.85,-3.35) arc (-100:100:.2cm and .35cm);					\draw[dotted, thick] (.8,-3.35) arc (-100:100:.2cm and .35cm);
\draw [xshift=2cm](9.05,-3.35) arc (-100:100:.2cm and .35cm);				\draw[dotted, thick] [xshift=2cm](10.7,-3.35) arc (-100:100:.2cm and .35cm);
\draw (-.8,-2.8) arc (80:280:.1cm and .2cm);							\draw[dotted, thick] (.8,-2.8) arc (80:280:.1cm and .2cm);
\draw [xshift=2cm](9.1,-2.8) arc (80:280:.1cm and .2cm);							\draw[dotted, thick] [xshift=2cm](10.7,-2.8) arc (80:280:.1cm and .2cm);
\draw (-3.4,1.8) .. controls (-2.8, 1.8) and (-2.4, 2) .. (-2.1,2.65);\draw (3.4,1.8) .. controls (2.8, 1.8) and (2.4, 2) .. (2.1,2.65);
\draw [xshift=2cm](6.5,1.8) .. controls (7.1, 1.8) and (7.5, 2) .. (7.8,2.65);\draw [xshift=2cm](13.3,1.8) .. controls (12.7, 1.8) and (12.3, 2) .. (12,2.65);
\draw (-2.1,2.65) ..controls (-1.6,3.5) and (-1.2,3.5) .. (-.8,3.5);\draw (2.1,2.65) ..controls (1.6,3.5) and (1.2,3.5) .. (.8,3.5);
\draw [xshift=2cm](7.8,2.65) ..controls (8.3,3.5) and (8.7,3.5) .. (9.1,3.5);\draw[xshift=2cm] (12,2.65) ..controls (11.5,3.5) and (11.1,3.5) .. (10.7,3.5);
\draw (-3.4,-1.8) .. controls (-2.8,-1.8) and (-2.4,-2) ..(-2.1,-2.7);\draw (3.4,-1.8) .. controls (2.8,-1.8) and (2.4,-2) ..(2.1,-2.7);
\draw [xshift=2cm](6.5,-1.8) .. controls (7.1,-1.8) and (7.5,-2) ..(7.8,-2.7);\draw [xshift=2cm](13.3,-1.8) .. controls (12.7,-1.8) and (12.3,-2) ..(12,-2.7);
\draw (-2.1,-2.7) .. controls (-1.6,-3.5) and (-1.2,-3.5) .. (-.8,-3.6);\draw (2.1,-2.7) .. controls (1.6,-3.5) and (1.2,-3.5) .. (.8,-3.6);
\draw [xshift=2cm](7.8,-2.7) .. controls (8.3,-3.5) and (8.7,-3.5) .. (9.1,-3.6);\draw [xshift=2cm](12,-2.7) .. controls (11.5,-3.5) and (11.1,-3.5) .. (10.7,-3.6);
\draw (-0.8,0.5) .. controls (-1.1,0.5) and (-1.1,-0.4) .. (-0.8, -0.4);\draw (0.8,0.5) .. controls (1.1,0.5) and (1.1,-0.4) .. (0.8, -0.4);
\draw [xshift=2cm](9.1,0.5) .. controls (8.8,0.5) and (8.8,-0.4) .. (9.1, -0.4);\draw[xshift=2cm] (10.7,0.5) .. controls (11,0.5) and (11,-0.4) .. (10.7, -0.4);
%
%
%
\draw[<->] (-.27,2) -- (.28,2); 
\draw[<->] (-.27,-2) -- (.28,-2); 
\draw[<-] (4.5,0) -- (5,0);
\draw[dashed] (5,0) -- (7,0);
\draw[->] (7,0) -- (7.5,0);
\draw[<->] (11.63,2) -- (12.18,2); 
\draw[<->] (11.63,-2) -- (12.18,-2); 
\draw[<->]	(-4.5,0)	.. controls (-5,0) and (-6,0) .. (-6, -2) 
							.. controls (-4,-10) and (16,-10) .. (18,-2) 
							.. controls (18,0) and (17,0) .. (16.5,0);
%
%
%
%
\filldraw (-4.1,-4) circle (2pt);
\draw[thick] (-4.1,-4) -- (0,-4);
\draw[->][thick] (-4.1,-4) -- (-2,-4);
\filldraw (0,-4) circle (2pt);
\draw[thick] (0,-4) -- (4.1,-4);
\draw[<-, thick] (2,-4) -- (4.1,-4);
\filldraw (4.1,-4) circle (2pt);
\draw[<-] (4.5,-4) -- (5,-4);
\draw[dashed] (5,-4) -- (7,-4);
\draw[->] (7,-4) -- (7.5,-4);
\filldraw (7.7,-4) circle (2pt);
\draw[thick] (7.7,-4) -- (11.8,-4);
\draw[->][thick] (7.7,-4) -- (9.7,-4);
\filldraw (11.8,-4) circle (2pt);
\draw[thick] (11.8,-4) -- (15.9,-4);
\draw[<-, thick] (13.8,-4) -- (15.9,-4);
\filldraw (15.9,-4) circle (2pt);
%
%
%
%
\draw (-3.5,2.5) node {$C_0$};
\draw (-1,4) node {$C_1$};
\draw (0,-6.5) node {$M'$};
\draw (-2,0) node {$M$};
\draw (6,6) node {$M_k$};
\draw (-4.1,-4.5) node {$C_0$};
\draw (0,-4.5) node {$C_1$};
\draw (4.1,-4.5) node {$C_0$};
%
%
%
\begin{scope}[yshift=-.3cm]
\draw (-4,4.5) .. controls (-4,5) and (-3,5) .. (-2,5)
				   .. controls (-2,5) and (4.5,5) .. (5,5)
				   .. controls (5.5,5) and (5.7,5.3) .. (6,5.5)
				   .. controls (6.3,5.3) and (6.5,5.1) .. (7,5)
				   .. controls (7.5,5) and (14,5) .. (14.5,5)
				   .. controls (15.5,5) and (16,5) .. (16,4.5);
\end{scope}
%
%
%
\begin{scope}[xshift=2.4cm, yshift=-3.1cm] 
\draw[scale=.41, rotate=180](-4,4.5) .. controls (-4,5) and (-3,5) .. (-2,5)
				   .. controls (-2,5) and (4.5,5) .. (5,5)
				   .. controls (5.5,5) and (5.7,5.3) .. (6,5.5)
				   .. controls (6.3,5.3) and (6.5,5.1) .. (7,5)
				   .. controls (7.5,5) and (14,5) .. (14.5,5)
				   .. controls (15.5,5) and (16,5) .. (16,4.5);
\end{scope}
\end{tikzpicture}

For any $\ell <k$, there are degree one maps $f\colon M_k \rightarrow M_{k-\ell}$ 
obtained by folding $\ell$ copies of $M'$ in $M_k$ along its boundary. These maps send the last $\ell+1$ copies of $M'$ inside $M_k$ to the last copy of $M'$ in $M_{k-\ell}$ as illustrated in the following picture for $\ell=2$:

\begin{tikzpicture}
[scale=.9]
\draw (-6,0) -- (-2,0);
\draw[dashed] (-2,0) -- (0,0);
\draw (0,0) -- (6,0);
\foreach \x in {-6,-4,-2,0,2,4,6}
\filldraw (\x,0) circle (2pt);
\foreach \y in {-5,-3,1,3,5}
\filldraw (\y,0) circle (1pt);
\foreach \z in {-6,-4,0,2,4}
\draw[->] (\z,0) -- (\z+.5,0);
\foreach \q in {-4,-2,2,4,6}
\draw[->] (\q,0) -- (\q-.5,0);
\draw (-8,.5) node {$M_k$};
\foreach \x in {-6,-4,0,2,4,6}
\draw (\x,.5) node {$C_0$};
\foreach \y in {-5,1,3,5}
\draw (\y,.5) node {$\tiny{C_1}$};
\draw (-6,-2) -- (-2,-2);
\draw[dashed] (-2,-2) -- (0,-2);
\draw (0,-2) -- (2,-2);
\foreach \x in {-6,-4,-2,0,2}
\filldraw (\x,-2) circle (2pt);
\foreach \y in {-5,-3,1}
\filldraw (\y,-2) circle (1pt);
\draw (2,-2) -- (0,-2.5) -- (2,-2.7);
\filldraw (0,-2.5) circle (2pt);
\filldraw (2,-2.7) circle (2pt);
\filldraw (1,-2.25) circle (1pt);
\filldraw (1,-2.6) circle (1pt);
\foreach \z in {-6,-4,0}
\draw[->] (\z,-2) -- (\z+.5,-2);
\draw[->] (0,-2.5) -- (.5,-2.375);
\draw[->] (0,-2.5) -- (.5,-2.55);
\foreach \q in {-4,-2,2}
\draw[->] (\q,-2) -- (\q-.5,-2);
\draw[->] (2,-2) -- (1.5,-2.125);
\draw[->] (2,-2.7) -- (1.5,-2.65);
\foreach \x in {-6,-4,0,2}
\draw (\x,-1.5) node {$C_0$};
\foreach \y in {-5,1}
\draw (\y,-1.5) node {$\tiny{C_1}$};
\draw (-6,-4) -- (-2,-4);
\draw[dashed] (-2,-4) -- (0,-4);
\draw (0,-4) -- (2,-4);
\foreach \x in {-6,-4,-2,0,2}
\filldraw (\x,-4) circle (2pt);
\foreach \y in {-5,-3,1}
\filldraw (\y,-4) circle (1pt);
\foreach \z in {-6,-4,0}
\draw[->] (\z,-4) -- (\z+.5,-4);
\foreach \q in {-4,-2,2}
\draw[->] (\q,-4) -- (\q-.5,-4);
\draw (-8,-4.5) node {$M_{k-2}$};
\foreach \x in {-6,-4,0,2}
\draw (\x,-4.5) node {$C_0$};
\foreach \y in {-5,1}
\draw (\y,-4.5) node {$\tiny{C_1}$};
\draw[->] (-8,0) -- (-8,-4);
\end{tikzpicture}

The induced representation of $\pi_1(M_k)$ obtained by the induced map on fundamental groups 
composed with the lattice embedding of $\pi_1(M_{k-\ell})$ in $\mathrm{Isom}(\mathbb{H}^n)$ 
has volume equal to the volume of $M_{k-\ell}$, that is $(k-\ell)/k$ times the volume of the maximal representation.

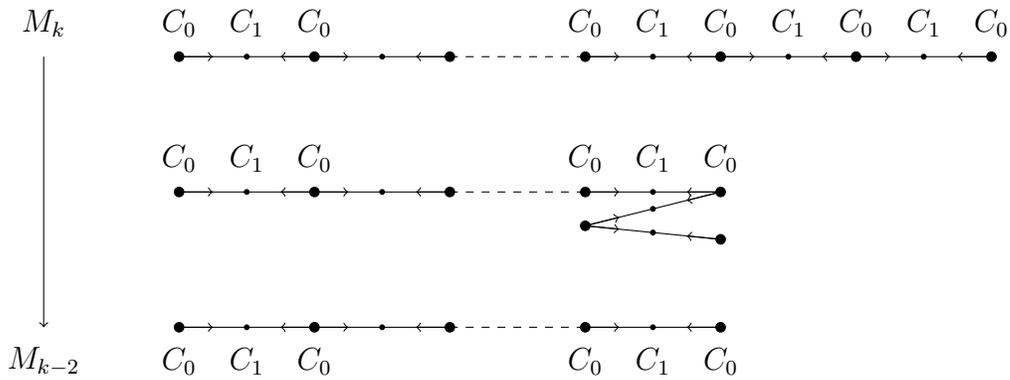
\appendix

\section{On the continuity of the volume of a representation}\label{sec:continuity}

The goal of this section is to prove the following:

\begin{prop}\label{thm:cont} Let $\Gamma<\isom(\HH^n)$ be any torsion-free lattice. 
The function 
\bqn
\ba
\hom(\Gamma, \mathrm{Isom}(\mathbb{H}^n))&\longrightarrow \quad\RR\\
\rho\qquad\quad&\longmapsto \Vol(\rho)
\ea
\eqn
is continuous.
\end{prop}

We begin with some preliminaries.
Let $G$ be a locally compact group, $\Gamma<G$ a lattice and $L$ a locally compact group. 
We denote by $\mathrm{C}_\mathrm{b}(X)$ the continuous bounded real valued functions on a topological space $X$. 
We define a map 
\bqn
\begin{array}{rccl}
&\mathrm{C}_\mathrm{b}(L^{n+1})\times \mathrm{Rep}(\Gamma,L)&\longrightarrow &\mathrm{C}_\mathrm{b}(\Gamma^{n+1})\\
&(c,\pi)&\longmapsto &\quad\pi^*( c ), 
\end{array}
\eqn
where 
\bqn
\pi^*( c )(\gamma_0,\dots,\gamma_n)=c(\pi(\gamma_0),\dots,\pi(\gamma_n))\,.
\eqn
We endow $\mathrm{C}_\mathrm{b}(L^{n+1})$ with the topology of uniform convergence and 
$\mathrm{C}_\mathrm{b}(\Gamma^{n+1})$ with the topology of pointwise convergence with control of norms:
in other words $\alpha_n\to\alpha$ in $\mathrm{C}_\mathrm{b}(\Gamma^{n+1})$ 
if it converges pointwise and $\sup_n\|\alpha_n\|_\infty<\infty$.
With these topologies and with the pointwise convergence topology on $\mathrm{Rep}(\Gamma,L)$
the above map is continuous.

We proceed to implement the transfer from $\Gamma$ to $G$. 
For this, let $s\colon \Gamma\setminus G\rightarrow G$ be a Borel section and $r\colon G\rightarrow \Gamma$ be defined by
\bqn
g=r(g)\cdot s(p(g))\,,
\eqn
where $p\colon G\rightarrow \Gamma\setminus G$ denotes the canonical projection. 
Given a $\Gamma$-invariant cochain $\alpha\in \mathrm{C}_\mathrm{b}(\Gamma^{n+1})^\Gamma$, 
define
\bqn
T\alpha(g_0,\dots,g_n)=\int_{\Gamma\setminus G} \alpha(r(gg_0),\dots,r(gg_n))d\mu(g)\,,
\eqn
where $\mu$ is the Haar measure on $G$ normalized so that $\mu(\Gamma\setminus G)=1$. 

\begin{prop} Suppose that the Borel section has the property that images of compact subsets are precompact. Then
\begin{enumerate}
\item $T\alpha$ is continuous, hence $T\alpha\in \mathrm{C}_\mathrm{b}(G^{n+1})^G$,
\item $T\colon  \mathrm{C}_\mathrm{b}(\Gamma^{n+1})^\Gamma \rightarrow \mathrm{C}_\mathrm{b}(G^{n+1})^G$ 
is continuous for pointwise convergence on $\Gamma$ 
with control of norms and uniform convergence on compact sets on $G^{n+1}$.
\end{enumerate}
\end{prop}

\begin{proof} Let $D=s(\Gamma\setminus G)$. Choose $\epsilon>0$ and $C\subset D$ compact with $\mu(D\setminus C)<\epsilon$. Then
\bqn
\mid T\alpha(g_0,\dots,g_n)-\int_C \alpha(r(gg_0),\dots,r(gg_n))d\mu(g) \mid < \epsilon \| \alpha \|_\infty\,.
\eqn
Now we write 
\bqn
\int_C \alpha(r(gg_0),\dots,r(gg_n))d\mu(g) = \sum_{\gamma_0,\dots,\gamma_n}\alpha(\gamma_0,\dots,\gamma_n)\mu(C\cap \gamma_0Dg_0^{-1}\cap \dots \cap \gamma_n Dg_n^{-1})\,.
\eqn

Before we continue with the proof of the proposition, 
we need to show that for every compact subset $K\subset G$ the number $F_K$ of translates of the fundamental domain $D$ 
that $K$ intersects is finite: 

\begin{lemma} For any compact subset $K\subset G$, the set
\bqn
F_K:= \{ \gamma \in \Gamma \mid K \cap \gamma D\neq \emptyset\}
\eqn
is finite.  
\end{lemma}

Note that the lemma is wrong for arbitrary fundamental domains, even for cocompact $\Gamma$. 
Indeed, start by writing the standard fundamental domain $(0,1]$ of $\ZZ$ in $\RR$ as
\bqn
D_0=\sqcup_{n=1}^{+\infty} (1/2^n,1/2^{n-1}]\,,
\eqn
and perturb it by translating each of the disjoint interval of $D_0$ by a different translation, 
for example obtaining the new fundamental domain
\bqn
D=\sqcup_{n=1}^{+\infty} n+(1/2^n,1/2^{n-1}]\,.
\eqn
Take as compact set the closed interval $C=[0,1]$. Then for every $-n\leq 0$, the intersection $C\cap (-n+D)$ is nonempty. 

\begin{proof} Set $F:=\cup_{\eta\in \Gamma} \eta K$ and 
observe that $F\cap D=s(p(F))$ is relatively compact by our choice of Borel section. 
Since $\gamma K \cap D= \gamma K \cap (F\cap D)$ and $K$ and $F\cap D$ are relatively compact, 
the lemma follows by the discreteness of $\Gamma$. 
\end{proof}

Going back to the proof of the proposition, fix compact subsets $C_0,\dots,C_n$ of $G$ such that $g_i\in C_i$. 
Observe that $F_{Cg_i}\subset F_{CC_i}$ and if $\gamma_i\in F_{CC_i}\setminus F_{Cg_i}$ 
then the measure of $C\cap \gamma_0Dg_0^{-1}\cap \dots \cap \gamma_n Dg_n^{-1}$ is zero. 
We can thus rewrite the above sum as 
\bqn
\int_C \alpha(r(gg_0),\dots,r(gg_n))d\mu(g) 
= \sum_{\gamma_i\in F_{CC_i}}\alpha(\gamma_0,\dots,\gamma_n)\mu(C\cap \gamma_0Dg_0^{-1}\cap \dots \cap \gamma_n Dg_n^{-1})\,,
\eqn
for any $(g_0,\dots,g_n)\in C_0\times \dots C_n$. 

The point (2) of Proposition follows since if $\alpha_n\rightarrow \alpha$ with pointwise convergence and 
$\sup_n \| \alpha_n\|_\infty < +\infty$ then $T\alpha_n\rightarrow T\alpha$ uniformly on compact sets.

\medskip
Finally, we show (1) by showing that the function 
\bqn
(g_0,\dots,g_n)\mapsto \mu(\{ C\cap \gamma_0Dg_0^-1\cap \dots \cap \gamma_n Dg_n^{-1}\})
\eqn
is continuous. 
To estimate the difference
\bqn
 \mu(C\cap \bigcap_{i=1}^n \gamma_i D g_i^{-1} )-\mu(C\cap \bigcap_{i=1}^n \gamma_i D h_i^{-1})
\eqn
we introduce the notation
\bqn
A(x_0,\dots ,x_n):=C \cap \bigcap_{i=0}^n \gamma_i D x_i^{-1} \,,
\eqn
for any $x_0,\dots, x_n\in G$. The above difference thus becomes
\bqn
\mu(A(g_0,\dots,g_n))-\mu(A(h_0,\dots,h_n))
\eqn
which we rewrite as a telescopic sum
\bqn
\sum_{i=0}^n\left( \mu(A(h_0,\dots,h_{i-1},g_{i},g_{i+1},\dots g_n))-\mu(A(h_0,\dots,h_{i-1},h_i,g_{i+1},\dots g_n))\right)\,.
\eqn
Setting 
\bqn
B_j:=C \cap \bigcap_{\ell=0}^{i-1} \gamma_\ell D h_\ell^{-1} \cap \bigcap_{\ell=i+1}^n \gamma_\ell D g_\ell^{-1}\,,
\eqn
the telescopic sum becomes
\bqn
\sum_ {j=0}^n (\mu(B_j\cap \gamma_jDg_j^{-1})-\mu(B_j\cap \gamma_j Dh_j^{-1}))\,.
\eqn
Using the simple set theoretical inequality valid for any sets $B,E,E'$
\bqn
|\mu(B\cap E)-\mu(B\cap E')| \leq \mu((B\cap E)\Delta (B\cap E'))\leq \mu(E\Delta E')\,,
\eqn
we obtain for each summand the estimate
\bqn
|\mu(B_j\cap \gamma_jDg_j^{-1})-\mu(B_j\cap \gamma_j Dh_j^{-1})|\leq \mu(  \gamma_jDg_j^{-1}\Delta  \gamma_jDh_j^{-1})=\| \chi_{Dg_j^{-1}h_j}-\chi_D\|_1\,.
\eqn
Thus
\bqn
\ba
&| \mu(C\cap \bigcap_{i=1}^n \gamma_i D g_i^{-1} )-\mu(C\cap \bigcap_{i=1}^n \gamma_i D h_i^{-1})|\\
\leq &\sum_{j=0}^n |\mu(B_j\cap \gamma_jDg_j^{-1})-\mu(B_j\cap \gamma_j Dh_j^{-1})|\\
\leq &\sum_{j=0}^n \| \chi_{Dg_j^{-1}h_j}-\chi_D\|_1\,.
\ea
\eqn
The continuity of the right regular action of $G$ on $L^1(G)$ concludes the proof of the proposition.
\end{proof}


\begin{proof}[Proof of Proposition~\ref{thm:cont}] Consider $\Gamma$ as a lattice in the full isometry group $\mathrm{Isom}(\mathbb{H}^n)$ and 
denote by $\varepsilon\colon \mathrm{Isom}(\mathbb{H}^n)\rightarrow \{-1,+1\}$ 
the homomorphism sending an isometry to $+1$ if it preserves orientation and $-1$ otherwise. 
By what precedes, the cohomology class $\mathrm{transf}(\rho^*(\omega_{\mathbb{H}^n}))$ 
can be represented by the continuous cocycle sending 
$(g_0,\dots,g_n)\in \mathrm{Isom}(\mathbb{H}^n)^{n+1}$ to
\bqn
\int_{\Gamma\setminus \mathrm{Isom}(\mathbb{H}^n)} \varepsilon(g) \omega_n(\rho(r(gg_0)),\dots,\rho(r(gg_n)))d\mu(g)\,.
\eqn
Note that the cocycle stays continues after transferring from $\mathrm{Isom}^+(\mathbb{H}^n)$ to $\mathrm{Isom}(\mathbb{H}^n)$. 
Integrating over a maximal compact subgroup $K$ in $\mathrm{Isom}(\mathbb{H}^n)$ 
we obtain a continuous cocycle $(\mathbb{H}^n)^{n+1}\rightarrow \RR$ 
that sends an $(n+1)$-tuple of points $g_0K,\dots,g_nK\in \mathrm{Isom}(\mathbb{H}^n)/K\cong \mathbb{H}^n$ to
\begin{equation}\label{ eq: integral for cocycle} 
\int_{K^{n+1}}\prod_{i=0}^ndk_i \int_{\Gamma\setminus \mathrm{Isom}(\mathbb{H}^n)} \varepsilon(g) \omega_n(\rho(r(gg_0k_0)),\dots,\rho(r(gg_nk_n)))d\mu(g)\,.
\end{equation}

We showed in \cite[Proposition 3.3]{Bucher_Burger_Iozzi_Mostow} that
\bqn 
\mathrm{transf}(\rho^*(\omega_n))=\frac{\Vol(\rho)}{\Vol(M)}\cdot \omega_{\mathbb{H}^n}\in \hcb^n(\mathrm{Isom}(\mathbb{H}^n),\RR_\varepsilon)\,.
\eqn 
Since there are no coboundaries in degree $n$ for $\mathrm{Isom}(\mathbb{H}^n)$-equivariant continuous bounded cochains on $\mathbb{H}^n$, 
this implies that we have a strict equality between (\ref{ eq: integral for cocycle}) and
\bqn
\frac{\Vol(\rho)}{\Vol(M)} \cdot \omega_n(g_0K,\dots,g_nK)\,.
\eqn
Since (\ref{ eq: integral for cocycle}) varies continuously in $\rho$, so does  $\Vol(\rho)$. \end{proof}

\vskip1cm
\bibliographystyle{alpha}

\vskip1cm

\newcommand{\etalchar}[1]{$^{#1}$}

 \end{document}